\documentclass[12pt]{amsart}

\usepackage[utf8]{inputenc}
\usepackage[T1]{fontenc}

\usepackage{amsmath,amssymb,amsfonts,textcomp,amsthm,xifthen,graphicx,color}

\usepackage{multirow}

\usepackage{enumerate}
\usepackage{enumitem}

\usepackage{fullpage}

\usepackage[utf8]{inputenc}

\usepackage[pdftex,
            pdfauthor={Thomas F\"uhrer},
            pdftitle={On a mixed FEM and a FOSLS with $H^{-1}$ loads},
            ]{hyperref}

\usepackage{pgfplots}
\usepgfplotslibrary{groupplots}
\pgfplotsset{compat=newest}
\usepgfplotslibrary{external}
\usepgfplotslibrary{colorbrewer}
\newtheorem{theorem}{Theorem}
\newtheorem{lemma}[theorem]{Lemma}
\newtheorem{example}[theorem]{Example}
\newtheorem{corollary}[theorem]{Corollary}
\newtheorem{proposition}[theorem]{Proposition}

\newcommand{\patch}{\omega}
\newcommand{\Patch}{\Omega}

\newcommand{\conv}{\mathrm{conv}}
\newcommand{\weight}{\boldsymbol{\varphi}}
\newcommand{\JSZ}{J^{\mathrm{SZ}}}
\newcommand{\JCL}{J^{\mathrm{Cl}}}
\newcommand{\QCL}{Q^{\mathrm{Cl}}}
\newcommand{\JCLw}{J^{\mathrm{Cl},\weight}}
\newcommand{\QCLw}{Q^{\mathrm{Cl},\weight}}
\newcommand{\Bop}{B}

\def\di{\mathrm{d}}

\DeclareMathOperator{\supp}{supp}

\DeclareMathOperator*{\argmin}{arg\, min}

\newcommand{\ip}[2]{(#1\hspace*{.5mm},#2)}

\newcommand{\norm}[3][]{#1\|#2#1\|_{#3}}

\newcommand{\diam}{\mathrm{diam}}

\def\div{\operatorname{div}}

\newcommand{\Hdivset}[1]{\boldsymbol{H}(\div;#1)}

\newcommand{\set}[2]{\big\{#1\,:\,#2\big\}}

\newcommand{\RT}{\ensuremath{\mathcal{RT}}}

\newcommand{\R}{\ensuremath{\mathbb{R}}}
\newcommand{\N}{\ensuremath{\mathbb{N}}}

\newcommand{\TT}{\ensuremath{\mathcal{T}}}

\newcommand{\PP}{\ensuremath{\mathcal{P}}}

\newcommand{\OO}{\ensuremath{\mathcal{O}}}

\newcommand{\VV}{\ensuremath{\mathcal{V}}}

\newcommand{\ssigma}{{\boldsymbol\sigma}}

\newcommand{\ttau}{{\boldsymbol\tau}}

\newcommand{\cchi}{{\boldsymbol\chi}}

\begin{document}

\title{On a mixed FEM and a FOSLS with $H^{-1}$ loads}
\date{\today}

\author{Thomas F\"uhrer}
\address{Facultad de Matem\'{a}ticas, Pontificia Universidad Cat\'{o}lica de Chile, Santiago, Chile}
\email{tofuhrer@mat.uc.cl}

\thanks{{\bf Acknowledgment.} 
This work was supported by ANID through FONDECYT project 1210391.
The author thanks Rami Masri (Brown University) and coauthors for pointing out that the claim made in earlier versions of this manuscript, that the coefficients in~\eqref{eq:comb} can always be chosen non-negative, is false in general. The present version does not rely on non-negativity.}

\keywords{least-squares method, mixed FEM, singular data}
\subjclass[2010]{65N30, 
                 65N12 
                 }
\begin{abstract}
We study variants of the mixed finite element method (mixed FEM) and the first-order system least-squares finite element (FOSLS) for the Poisson problem where we replace the load by a suitable regularization which permits to use $H^{-1}$ loads. We prove that any bounded $H^{-1}$ projector onto piecewise constants can be used to define the regularization and yields quasi-optimality of the lowest-order mixed FEM resp. FOSLS in weaker norms. Examples for the construction of such projectors are given. One is based on the adjoint of a weighted Cl\'ement quasi-interpolator. We prove that this Cl\'ement operator has second-order approximation properties. For the modified mixed method we show optimal convergence rates of a postprocessed solution under minimal regularity assumptions --- a result not valid for the lowest-order mixed FEM without regularization. Numerical examples conclude this work.
\end{abstract}
\maketitle

\section{Introduction}
In this work we study a mixed finite element method (FEM) and a first-order least-squares FEM (FOSLS) for the Poisson problem with $H^{-1}(\Omega)$ loads where $H^{-1}(\Omega)$ denotes the topological dual of the Sobolev space $H_0^1(\Omega)$ and $\Omega\subset \R^n$ ($n=2,3$) denotes a bounded Lipschitz domain with polytopal boundary.
Both numerical methods are based on the following first-order reformulation of the Poisson problem with homogeneous Dirichlet boundary conditions, 
\begin{subequations}\label{eq:fo}
  \begin{alignat}{2}
    \div\ssigma &= -f &\quad&\text{in }\Omega,\label{eq:fo:a} \\
    \ssigma-\nabla u &= 0 &\quad&\text{in }\Omega, \label{eq:fo:b}\\
    u &= 0 &\quad&\text{on }\Gamma:=\partial\Omega.
  \end{alignat}
\end{subequations}
Given $f\in L^2(\Omega)$ the mixed FEM seeks $(u_\TT,\ssigma_\TT)\in W_\TT\times \Sigma_\TT$ such that 
\begin{subequations}\label{eq:mixed:standard}
  \begin{alignat}{2}
    &\ip{\ssigma_\TT}{\ttau} + \ip{u_\TT}{\div\ttau} &\,=\,& 0, \\
    &\ip{\div\ssigma_\TT}{v} &\,=\,&  \ip{-f}{v}
  \end{alignat}
\end{subequations}
for all $(v,\ttau)\in W_\TT\times \Sigma_\TT := \PP^0(\TT)\times \RT^0(\TT)$, where $\TT$ denotes a shape-regular mesh of simplices of $\Omega$, $\PP^p(\TT)$ denotes the space of $\TT$-piecewise polynomials of degree less or equal to $p\in\N_0$ and $\RT^0(\TT)$ is the lowest-order Raviart--Thomas space.
Given $f\in L^2(\Omega)$ the FOSLS seeks the minimizer of the $L^2(\Omega)$ residuals of~\eqref{eq:fo:a}--\eqref{eq:fo:b} over the discrete space $U_\TT\times \Sigma_\TT = \PP^1(\TT)\cap H_0^1(\Omega) \times \Sigma_\TT$, i.e., 
\begin{align}\label{eq:fosls:standard}
  (u_\TT,\ssigma_\TT) = \argmin_{(v,\ttau)\in U_\TT\times \Sigma_\TT} \norm{\div\ttau+f}{}^2 + \norm{\nabla v-\ttau}{}^2.
\end{align}

Both methods, \eqref{eq:mixed:standard} and~\eqref{eq:fosls:standard}, are not well defined for $f\in H^{-1}(\Omega)$. In the recent article~\cite{MINRESsingular} we proposed to replace the load in~\eqref{eq:fosls:standard} by a suitable polynomial approximation. The very same ideas in the analysis can be applied for the mixed FEM~\eqref{eq:mixed:standard}. 
To describe the new variants, let $Q_\TT^\star\colon H^{-1}(\Omega)\to \PP^0(\TT)\subseteq H^{-1}(\Omega)$ denote a bounded projection operator, i.e.,
\begin{align*}
  Q_\TT^\star \phi = \phi \quad\text{for all } \phi\in \PP^0(\TT) \qquad\text{and}\qquad 
  \norm{Q_\TT^\star\phi}{-1} \lesssim \norm{\phi}{-1} \quad\text{for all }\phi \in H^{-1}(\Omega).
\end{align*}
  
The modified methods are defined as follows: 
\newline
\emph{Modified mixed FEM:} Given $f\in H^{-1}(\Omega)$ seek $(u_\TT,\ssigma_\TT)\in W_\TT\times \Sigma_\TT$ such that
\begin{subequations}\label{eq:mixed:mod}
  \begin{alignat}{2}
    &\ip{\ssigma_\TT}{\ttau} + \ip{u_\TT}{\div\ttau} &\,=\,& 0, \\
    &\ip{\div\ssigma_\TT}{v} &\,=\,&  \ip{-Q_\TT^\star f}{v}
  \end{alignat}
\end{subequations}
for all $(v,\ttau)\in W_\TT\times \Sigma_\TT$.
\newline
\emph{Modified FOSLS:} Given $f\in H^{-1}(\Omega)$ solve
\begin{align}\label{eq:fosls:mod}
  (u_\TT,\ssigma_\TT) = \argmin_{(v,\ttau)\in U_\TT\times \Sigma_\TT} \norm{\div\ttau+Q_\TT^\star f}{}^2 + \norm{\nabla v-\ttau}{}^2.
\end{align}

In our recent work~\cite{MINRESsingular} we proved that the solution $(u_\TT,\ssigma_\TT)$ of~\eqref{eq:fosls:mod} satisfies the error estimate
\begin{align*}
  \norm{u-u_\TT}{1} + \norm{\ssigma-\ssigma_\TT}{} \lesssim h^s\norm{f}{-1+s}
\end{align*}
where $s\in[0,1]$ depends on $\Omega$ and the regularity of $f$, and $h$ is the maximum element diameter.
A similar estimate may be derived for the solution of~\eqref{eq:mixed:mod} following the techniques from~\cite{MINRESsingular}, or the ones presented here, see Corollary~\ref{cor:quasiopt:mixed} below.

In the article at hand we complement on our results from~\cite{MINRESsingular} in that we show quasi-optimality of both the modified methods~\eqref{eq:mixed:mod} and~\eqref{eq:fosls:mod}: 
Let $(u,\ssigma)\in H_0^1(\Omega)\times L^2(\Omega;\R^n)$ denote the solution of~\eqref{eq:fo}. If $(u_\TT,\ssigma_\TT)\in W_\TT\times \Sigma_\TT$ denotes the solution of~\eqref{eq:mixed:mod}, then (see Theorem~\ref{thm:quasiopt:mixed})
\begin{align*}
  \norm{u-u_\TT}{} + \norm{\ssigma-\ssigma_\TT}{} \lesssim \min_{(v,\ttau)\in W_\TT\times \Sigma_\TT} \norm{u-v}{} + \norm{\ssigma-\ttau}{},
\end{align*}
or if $(u_\TT,\ssigma_\TT)\in U_\TT\times \Sigma_\TT$ denotes the solution of~\eqref{eq:fosls:mod}, then (see Theorem~\ref{thm:quasiopt:fosls})
\begin{align*}
  \norm{u-u_\TT}{1} + \norm{\ssigma-\ssigma_\TT}{} \lesssim \min_{(v,\ttau)\in U_\TT\times\Sigma_\TT} \norm{u-v}{1} + \norm{\ssigma-\ttau}{}.
\end{align*}

The mixed FEM~\eqref{eq:mixed:standard} and FOSLS~\eqref{eq:fosls:standard} have been studied thoroughly and we refer the interested reader to~\cite{BoffiBrezziFortin,Gatica14,BochevGunzberger09} for an introduction, overview and further literature on these methods. 
A variant of the hybrid higher-order method (known as HHO) with $H^{-1}(\Omega)$ loads is introduced and analyzed in~\cite{ErnZanotti20}. 
For further details on and constructions of different $H^{-1}(\Omega)$ projection operators onto piecewise polynomial spaces we refer to the recent work~\cite{DS21intOp} where also various applications are discussed. In~\cite{MillarMugaRojasVdZee22} a general theory for the approximation of rough linear functionals is developed.

Postprocessing schemes for the mixed method~\eqref{eq:mixed:standard} are well known~\cite{Stenberg91}, and optimal convergence rates for higher-order elements can be shown, whereas the lowest-order case, as considered here, requires sufficiently regular solutions, see, e.g.,~\cite[Theorem~2.1 and Remark~2.1]{Stenberg91}. 
In the work at hand we prove that the postprocessing scheme from~\cite{Stenberg91} applied to solutions of the modified mixed FEM~\eqref{eq:mixed:mod} yields optimal rates with only minimal regularity assumptions. 

The analysis of the latter is based on the dual of a weighted Cl\'ement quasi-interpolator. The advantage of our proposed construction is that the Cl\'ement operator reconstructs an approximation with second-order approximation properties from an elementwise projection on constants.
For an overview on Cl\'ement quasi-interpolators we refer to the works~\cite{Clement75,CarstensenClement99} and for additional information and applications to~\cite{CarstensenClement}.
As a side product of our analysis we obtain a result on the approximation by piecewise constants in the dual space of $H^2(\Omega)\cap H_0^1(\Omega)$ (Corollary~\ref{cor:Hm2approx}).

Our results on quasi-optimality in weaker norms might also be of interest for the analysis of FOSLS for eigenvalue problems~\cite{BertrandBoffi22}.
The authors of~\cite{MRV22} define a superconvergent FEM based on the postprocessing technique from~\cite{Stenberg91}.
Our new findings for the postprocessing scheme (Section~\ref{sec:L2}) could also improve the results from~\cite{MRV22} for the lowest-order case.

In this article we only consider lowest-order discretizations, though, many results can be extended to the higher-order case. E.g., for the FOSLS we refer the reader to the very recent work~\cite[Remark~4.7]{MSS23}.
We restrict the presentation to $n=2,3$ but note that our results are valid for $n=1$. 
The remainder of this work is organized as follows:
Section~\ref{sec:main} introduces some notation and contains the statement and proofs of the quasi-optimality results stated above. 
In Section~\ref{sec:clement} we study a weighted Cl\'ement quasi-interpolator and discuss some of its main properties. 
Optimal error estimates for the postprocessed solution of~\eqref{eq:mixed:mod} and optimal $L^2(\Omega)$ error estimates for the scalar solution of~\eqref{eq:fosls:mod} are given in Section~\ref{sec:L2}. 
This article closes with various numerical experiments (Section~\ref{sec:ex}).

\section{Quasi-optimality}\label{sec:main}
This section is devoted to the proof of quasi-optimality results of the modified variants~\eqref{eq:mixed:mod} and~\eqref{eq:fosls:mod} of the mixed FEM and FOSLS claimed in the introduction. 
Before we give details in Section~\ref{sec:quasiopt} we recall some known properties of projection operators needed for the analysis in Section~\ref{sec:projections}.
The proof of quasi-optimality requires a $H^{-1}(\Omega)$ bounded projection operator and we also give an example of such an operator that is easy to implement.

\subsection{Sobolev spaces}
For a Lipschitz domain $K\subseteq \Omega$ we denote by $H^k(K)$, $H_0^k(K)$, $k\in \N$ the usual Sobolev spaces with norms $\norm{\cdot}{K,k}$. If $K=\Omega$ we simply write $\norm{\cdot}{k}$. The space $H_0^1(\Omega)$ is equipped with the norm $\norm{\cdot}1:=\norm{\nabla(\cdot)}{}$ where $\norm{\cdot}{}$ is the $L^2(\Omega)$ norm with inner product $\ip{\cdot}{\cdot}$. 
Similarly, $\norm{\cdot}{K}$ is the $L^2(K)$ norm with inner product $\ip{\cdot}{\cdot}_K$.
Intermediate Sobolev spaces with index $s$ are defined by (real) interpolation, e.g., $\widetilde H^s(\Omega) = [L^2(\Omega),H_0^1(\Omega)]_s$, $H^s(\Omega) = [L^2(\Omega),H^1(\Omega)]_s$ with norm denoted by $\norm{\cdot}s$. 
Dual spaces of Sobolev spaces are understood with respect to the extended $L^2$ inner product, e.g, the dual of $H_0^1(\Omega)$ is denoted by $H^{-1}(\Omega)$ and equipped with the dual norm
\begin{align*}
  \norm{\phi}{-1} = \sup_{0\neq v\in H_0^1(\Omega)} \frac{\ip{\phi}v}{\norm{v}1}.
\end{align*}
Note that $H^{-s}(\Omega) = (\widetilde H^s(\Omega))'$ with norms $\norm{\cdot}{-s}$, $s\in(0,1)$.

\subsection{Projection and interpolation operators}\label{sec:projections}
Let $\TT$ denote a regular mesh of simplices of $\Omega$ with $h_\TT\in L^\infty(\Omega)$ denoting the elementwise mesh-size function, $h_\TT|_T = \diam(T)$ for $T\in\TT$.
With $\VV$ we denote all vertices of $\TT$ and $\VV_0 = \VV\setminus\Gamma$ are the interior vertices. 
The set of $n+1$ vertices of an element $T\in\TT$ is $\VV_T$.
The patch of all elements of $\TT$ sharing a node $z\in\VV$ is denoted by $\patch_z$ and $\Patch_z$ is used for the domain associated to $\patch_z$. 
The element patch $\patch_T$ is the union of all vertex patches $\patch_z$ with $z\in\VV_T$ and $\Patch_T$ is the corresponding domain.

Let $\Pi_\TT^0\colon L^2(\Omega)\to W_\TT=\PP^0(\TT)$ denote the $L^2(\Omega)$-orthogonal projection which has the first-order approximation property
\begin{align*}
  \norm{(1-\Pi_\TT^0)\phi}{T} \lesssim h_T\norm{\nabla \phi}{T} \quad\text{for all } \phi\in H^1(T), \,T\in\TT. 
\end{align*}
Here, and in the remainder, the notation $A\lesssim B$ means that there exists a generic constant, possibly depending on the shape-regularity constant $\kappa_\TT$ and $\Omega$, such that $A\leq C\cdot B$. The notation $A\eqsim B$ means $A\lesssim B$ and $B\lesssim A$.
The shape-regularity constant of a mesh $\TT$ is given by
\begin{align*}
  \kappa_\TT = \max_{T\in\TT} \frac{h_T^n}{|T|},
\end{align*}
where $|\cdot|$ denotes the volume measure.

Recall that $\Sigma_\TT = \RT^0(\TT)$ is the lowest-order Raviart--Thomas space. We denote by $\Pi_\TT^{\RT}\colon \Hdivset\Omega\to \Sigma_\TT$ the projector constructed in~\cite{egsv22}. It has the following properties, see~\cite[Theorem~3.2]{egsv22}, where $\RT^0(T)$ denotes the lowest-order Raviart--Thomas space on the element $T$.
\begin{subequations}\label{eq:propRTproj}
\begin{align}
  \div\Pi_\TT^{\RT}\ssigma &= \Pi_\TT^0\div\ssigma \quad\text{and}\\
  \norm{\ssigma-\Pi_\TT^{\RT}\ssigma}{}^2 &\lesssim \sum_{T\in\TT}\min_{\ttau\in \RT^0(T)} \norm{\ssigma-\ttau}{T}^2 + \norm{h_\TT(1-\Pi_\TT^0)\div\ssigma}{}^2
\end{align}
and, in particular, 
\begin{align}
  \norm{\Pi_\TT^{\RT}\ssigma}{\Hdivset\Omega}^2 := \norm{\Pi_\TT^{\RT}\ssigma}{}^2  +
  \norm{\div\Pi_\TT^{\RT}\ssigma}{}^2 \lesssim \norm{\ssigma}{\Hdivset\Omega}^2
\end{align}
for all $\ssigma\in \Hdivset\Omega$.
\end{subequations}

There are several possibilities to construct a bounded projection $Q_\TT^\star\colon H^{-1}(\Omega)\to \PP^0(\TT)$. 
We refer the interested reader to~\cite{DS21intOp} for an overview on existing operators and the construction of a family of $H^{-1}(\Omega)$ projectors into polynomial spaces. 
Here, we follow the construction presented in~\cite{MINRESsingular} resp.~\cite[Section~2.4]{MultilevelNorms21}. 
First, define the averaged Scott--Zhang-type quasi-interpolator $\JSZ_\TT\colon H_0^1(\Omega)\to\PP^1(\TT)\cap H_0^1(\Omega)$ by
\begin{align*}
  \JSZ_\TT v = \sum_{z\in\VV_0} \ip{v}{\psi_z} \eta_z, 
\end{align*}
where $\{\eta_z\}_{z\in\VV_0}$ is the nodal basis of $\PP^1(\TT)\cap H_0^1(\Omega)$, i.e., $\eta_z(z') = \delta_{z,z'}$ for all $z,z'\in \VV_0$, and $\delta_{z,z'}$ denotes the Kronecker-$\delta$ symbol. 
Furthermore, $\psi_z \in \PP^1(\TT)$ with $\supp\psi_z\subseteq \overline{\Patch_z}$ denotes the bi-orthogonal dual basis function satisfying
\begin{align*}
  \ip{\psi_z}{\eta_{z'}} = \delta_{z,z'} \quad\forall z,z'\in \VV_0.
\end{align*}
An explicit representation is given by
\begin{align*}
  \psi_z|_{\Patch_z} = \frac{1}{|\Patch_z|}\big( (n+1)(n+2)\eta_z - (n+1)\big).
\end{align*}
For $T\in\TT$ define the bubble function $\eta_{\mathrm{b},T} := c_T \prod_{z\in \VV_T} \eta_z$ with $c_T$ chosen so that $\ip{\eta_{\mathrm{b},T}}1 = 1$ and
\begin{align*}
  \Bop_\TT v = \sum_{T\in\TT}\ip{v}{\chi_T}\eta_{\mathrm{b},T},
\end{align*}
where $\chi_T\in L^\infty(\Omega)$ denotes the characteristic function on $T\in\TT$.
It is straightforward to check that $\Bop_\TT$ is locally bounded, i.e.,
\begin{align*}
  \norm{\Bop_\TT v}T \lesssim \norm{v}T \quad\text{for all } T\in\TT \text{ and } v\in L^2(\Omega).
\end{align*}
The operator 
\begin{align}\label{eq:defQCL}
  \QCL_\TT := \Pi_\TT^0\big( \JSZ_\TT+\Bop_\TT(1-\JSZ_\TT)\big)' = \Pi_\TT^0\big( (\JSZ_\TT)'+(1-\JSZ_\TT)'\Bop_\TT'\big) \colon H^{-1}(\Omega)\to \PP^0(\TT)
\end{align}
with 
\begin{align*}
  (\JSZ_\TT)' \phi = \sum_{z\in\VV_0} \ip{\phi}{\eta_z} \psi_z, \quad (\Bop_\TT)' \phi = \sum_{T\in\TT} \ip{\phi}{\eta_{\mathrm{b},T}}\chi_T
\end{align*}
has the following properties.
\begin{proposition}[{\cite[Theorem~8]{MultilevelNorms21}}]\label{prop:QCL}
  The operator defined in~\eqref{eq:defQCL} satisfies
  \begin{itemize}
    \item $\QCL_\TT\phi = \phi$ for all $\phi\in \PP^0(\TT)$, 
    \item $\norm{(1-\QCL_\TT)\phi}{-1} \lesssim \norm{h_\TT\phi}{}$ for all $\phi\in L^2(\Omega)$, 
    \item $\norm{\QCL_\TT\phi}{-1}\lesssim \norm{\phi}{-1}$ for all $\phi \in H^{-1}(\Omega)$, and
    \item $\norm{\QCL_\TT\phi}{}\lesssim \norm{\phi}{}$ for all $\phi \in L^2(\Omega)$.
  \end{itemize}
\end{proposition}
In particular, $Q_\TT^\star = \QCL_\TT$ can be used in the modified schemes~\eqref{eq:mixed:mod},~\eqref{eq:fosls:mod}.
We further study this operator in Section~\ref{sec:clement} where we recall a relation of $\JSZ_\TT$ to the Cl\'ement interpolation operator with $0$-th order moments.
This relation is also of practical interest as it simplifies the calculation of $\QCL_\TT$.

\subsection{Analysis of the modified mixed FEM resp. FOSLS}\label{sec:quasiopt}
We need the following observation.
\begin{lemma}\label{lem:aux}
  Let $u\in H_0^1(\Omega)$. 
  If $Q_\TT^\star\colon H^{-1}(\Omega)\to W_\TT\subseteq H^{-1}(\Omega)$ is a bounded projector, then, 
  \begin{align*}
    \norm{(1-Q_\TT^\star)\Delta u}{-1} \lesssim \norm{\nabla u-\ttau}{} \quad\text{for all } \ttau \in \Sigma_\TT.
  \end{align*}
\end{lemma}
\begin{proof}
  Let $\ttau\in \Hdivset\Omega$ with $\div\ttau\in W_\TT$ be arbitrary. Since $Q_\TT^\star$ is a bounded projection we have that
  \begin{align*}
    \norm{(1-Q_\TT^\star)\Delta u}{-1} &= \norm{(1-Q_\TT^\star)\div(\nabla u-\ttau)}{-1} 
    \\ & \lesssim \norm{\div(\nabla u-\ttau)}{-1} \leq \norm{\nabla u-\ttau}{},
  \end{align*}
  where the last estimate follows from boundedness of $\div\colon L^2(\Omega;\R^n)\to H^{-1}(\Omega)$. 
\end{proof}

%

The following theorem is the first main result of this section.
\begin{theorem}\label{thm:quasiopt:mixed}
  Let $Q_\TT^\star\colon H^{-1}(\Omega)\to W_\TT\subseteq H^{-1}(\Omega)$ denote a bounded projection.
  Given $f\in H^{-1}(\Omega)$, let $(u,\ssigma)\in H_0^1(\Omega)\times L^2(\Omega;\R^n)$ denote the solution of~\eqref{eq:fo}.
  The unique solution $(u_\TT,\ssigma_\TT)\in W_\TT\times \Sigma_\TT$ of~\eqref{eq:mixed:mod} satisfies
  \begin{align*}
    \norm{u-u_\TT}{} + \norm{\ssigma-\ssigma_\TT}{} \lesssim \min_{(v,\ttau)\in W_\TT\times \Sigma_\TT} \norm{u-v}{} + \norm{\ssigma-\ttau}{}.
  \end{align*}
\end{theorem}
\begin{proof}
  Let $\widetilde u\in H_0^1(\Omega)$ denote the unique weak solution of
  \begin{align*}
    -\Delta \widetilde u = Q_\TT^\star f, \quad \widetilde u|_\Gamma = 0. 
  \end{align*}
  By the triangle inequality and $\norm{u}{1} = \norm{\Delta u}{-1}$ we have that
  \begin{align*}
    \norm{u-u_\TT}{} + \norm{\ssigma-\ssigma_\TT}{} &\leq \norm{u-\widetilde u}{} + \norm{\nabla u-\nabla\widetilde u}{}
    +\norm{\widetilde u-u_\TT}{} + \norm{\nabla\widetilde u-\ssigma_\TT}{}
    \\
    &\lesssim \norm{(1-Q_\TT^\star)f}{-1} + \norm{\widetilde u-u_\TT}{} + \norm{\nabla\widetilde u-\ssigma_\TT}{}.
  \end{align*}
  Employing the quasi-optimality of the mixed method with datum in $L^2(\Omega)$, see, e.g.~\cite{Gatica14}, we get that
  \begin{align*}
    \norm{\widetilde u-u_\TT}{} + \norm{\nabla\widetilde u-\ssigma_\TT}{} \lesssim 
    \min_{(v,\ttau)\in W_\TT\times \Sigma_\TT} \norm{\widetilde u-v}{} + \norm{\nabla\widetilde u-\ttau}{}
    +\norm{Q_\TT^\star f+\div\ttau}{}.
  \end{align*}
  Using the properties~\eqref{eq:propRTproj} of $\Pi_\TT^{\RT}$ and $\div\nabla\widetilde u = -Q_\TT^\star f\in W_\TT$, we see that 
  \begin{align*}
    \min_{\ttau\in \Sigma_\TT} \norm{\nabla\widetilde u-\ttau}{}^2+\norm{Q_\TT^\star f+\div\ttau}{}^2  
    &\lesssim \sum_{T\in\TT} \min_{\ttau\in\RT^0(T)}\norm{\nabla\widetilde u-\ttau}T^2 + \norm{h_\TT(1-\Pi_\TT^0)Q_\TT^\star f}{}^2 
    \\
    &\qquad +\norm{Q_\TT^\star f+\div\Pi_\TT^{\RT}\nabla\widetilde u}{}^2
    \\
    &\leq \min_{\ttau\in\Sigma_\TT}\norm{\nabla\widetilde u-\ttau}{}^2.
  \end{align*}
  Combining the estimates and using the triangle inequality as well as $\norm{u-\widetilde u}{1} = \norm{(1-Q_\TT^\star)f}{-1}$ we infer that
  \begin{align*}
    \norm{u-u_\TT}{} + \norm{\ssigma-\ssigma_\TT}{} \lesssim \norm{(1-Q_\TT^\star)f}{-1} + \min_{(v,\ttau)\in W_\TT\times \Sigma_\TT} \norm{u-v}{} + \norm{\ssigma-\ttau}{}.
  \end{align*}
  The proof is finished by applying Lemma~\ref{lem:aux}.
\end{proof}

For the scheme~\eqref{eq:mixed:standard} with $f\in L^2(\Omega)$ a quasi-best approximation in the form
\begin{align}\label{eq:L2quasioptSigma}
  \norm{\ssigma-\ssigma_\TT}{} = \min_{\ttau\in \RT^0(\TT), \,\div\ttau=\Pi_\TT^0\div\ssigma}{} \norm{\ssigma-\ttau}{}
\end{align}
where $\ssigma\in\Hdivset\Omega$ is the solution to~\eqref{eq:fo} and $(u_\TT,\ssigma_\TT)\in W_\TT\times \Sigma_\TT$ is the solution of~\eqref{eq:mixed:standard} is known, see, e.g.~\cite{BoffiBrezziFortin} or~\cite[Lemma~6.1]{egsv22}. 
Note that the infimum is taken over a restricted set. 
For the modified version of the mixed method we have the following variant.
\begin{theorem}
 Let $Q_\TT^\star\colon H^{-1}(\Omega)\to W_\TT\subseteq H^{-1}(\Omega)$ denote a bounded projection.
  Given $f\in H^{-1}(\Omega)$, let $(u,\ssigma)\in H_0^1(\Omega)\times L^2(\Omega;\R^n)$ denote the solution of~\eqref{eq:fo}.
  The unique solution $(u_\TT,\ssigma_\TT)\in W_\TT\times \Sigma_\TT$ of~\eqref{eq:mixed:mod} satisfies
  \begin{align*}
    \norm{\ssigma-\ssigma_\TT}{} \lesssim \min_{\ttau\in\Sigma_\TT} \norm{\ssigma-\ttau}{}.
  \end{align*}
\end{theorem}
\begin{proof}
  Let $\widetilde u\in H_0^1(\Omega)$ solve $-\Delta \widetilde u = Q_\TT^\star f$, $\widetilde\ssigma:=\nabla\widetilde u$ and note that with quasi-optimality~\eqref{eq:L2quasioptSigma} and~\eqref{eq:propRTproj}
  \begin{align*}
    \norm{\widetilde\ssigma-\ssigma_\TT}{} = \min_{\ttau\in \RT^0(\TT), \,\div\ttau=\Pi_\TT^0\div\widetilde\ssigma}{}
    \norm{\widetilde\ssigma-\ttau}{} \leq \norm{\widetilde\ssigma-\Pi_\TT^{\RT}\widetilde\ssigma}{} \lesssim \min_{\ttau\in \Sigma_\TT}\norm{\widetilde\ssigma-\ttau}{}.
  \end{align*}
  Employing the triangle inequality twice we see that
  \begin{align*}
    \norm{\ssigma-\ssigma_\TT}{} \lesssim \norm{\ssigma-\widetilde\ssigma}{} + \min_{\ttau\in \Sigma_\TT}\norm{\widetilde\ssigma-\ttau}{}
    \lesssim \norm{\ssigma-\widetilde\ssigma}{} + \min_{\ttau\in \Sigma_\TT}\norm{\ssigma-\ttau}{}.
  \end{align*}
  The proof is concluded by observing that $\norm{\ssigma-\widetilde\ssigma}{}=\norm{(1-Q_\TT^\star)f}{-1}$ and applying Lemma~\ref{lem:aux}.
\end{proof}

To derive convergence rates we require additional regularity of the load and the regularity shift of the Poisson problem. 
Suppose that $u\in H_0^1(\Omega)$ is the solution of~\eqref{eq:fo}. Then, by elliptic regularity, see, e.g.,~\cite{grisvard,dauge88}, there exists $s_\Omega\in(1/2,1]$ depending only on $\Omega$ such that
\begin{align}\label{eq:ellipticreg}
  \norm{u}{1+t} \lesssim \norm{\Delta u}{-1+t} 
\end{align}
for all $t\in[0,s_\Omega]$ with $\Delta u \in H^{-1+t}(\Omega)$.
\begin{corollary}\label{cor:quasiopt:mixed}
  Under the assumptions of Theorem~\ref{thm:quasiopt:mixed} suppose additionally that $f\in H^{-1+s}(\Omega)$ for some $s\in[0,1]$. 
  Then, 
  \begin{align*}
    \norm{u-u_\TT}{} + \norm{\ssigma-\ssigma_\TT{}}{} \lesssim h^{\min\{s,s_\Omega\}} \norm{f}{-1+\min\{s,s_\Omega\}}.
  \end{align*}
\end{corollary}
\begin{proof}
  Quasi-optimality of Theorem~\ref{thm:quasiopt:mixed} and approximation properties of $\Pi_\TT^0$ imply that
  \begin{align*}
    \norm{u-u_\TT}{} + \norm{\ssigma-\ssigma_\TT}{} &\lesssim h\norm{\nabla u}{} + \min_{\ttau\in\Sigma_\TT} \norm{\ssigma-\ttau}{}.
  \end{align*}
  For the first term note that $\norm{\nabla u}{}=\norm{f}{-1}$.
  The other term can be bounded as in~\cite[Theorem~15]{MINRESsingular}. We give the details for the sake of completeness. 
  Note that 
  \begin{align*}
    \norm{\ssigma-\ttau}{} \leq \norm{\ssigma-\widetilde\ssigma}{} + \norm{\widetilde\ssigma-\ttau} \quad\forall \ttau\in \Sigma_\TT,
  \end{align*}
  where $\widetilde\ssigma = \nabla\widetilde u$ and $\widetilde u\in H_0^1(\Omega)$ solves $\Delta\widetilde u = -Q_\TT^\star f$. 
  By the projection property we have $(1-Q_\TT^\star)\QCL_\TT f =0$.
  Then, 
  \begin{align*}
    \norm{\ssigma-\widetilde\ssigma}{} = \norm{(1-Q_\TT^\star)f}{-1} \lesssim\norm{(1-\QCL_\TT)f}{-1} \lesssim h^{\min\{s,s_\Omega\}}\norm{f}{-1+\min\{s,s_\Omega\}}.
  \end{align*}
  The last estimate follows by an interpolation argument and Proposition~\ref{prop:QCL}. Finally, using~\eqref{eq:propRTproj} and choosing $\ttau=\Pi_\TT^{\RT}\widetilde\ssigma$ we obtain
  \begin{align*}
    \norm{\widetilde\ssigma-\ttau}{}^2 \lesssim \sum_{T\in\TT}\min_{\cchi\in \RT^0(T)} \norm{\widetilde\ssigma-\cchi}{T}^2 \lesssim 
    \norm{\ssigma-\widetilde\ssigma}{}^2 + \norm{\ssigma-\Pi_\TT^0\ssigma}{}^2
  \end{align*}
  by noting that $\PP^0(T)^n\subseteq \RT^0(T)$. The first term on the right-hand side is estimated as before and for the remaining term we get together with approximation properties of piecewise constants and elliptic regularity that $\norm{\ssigma-\Pi_\TT^0\ssigma}{}\lesssim h^{\min\{s,s_\Omega\}}\norm{f}{-1+\min\{s,s_\Omega\}}$. This finishes the proof.
\end{proof}

Next, we analyze quasi-optimality of the modified FOSLS~\eqref{eq:fosls:mod} in weaker norms.
The proof of the following result is similar to the proof of Theorem~\ref{thm:quasiopt:mixed}. 
\begin{theorem}\label{thm:quasiopt:fosls}
  Let $Q_\TT^\star\colon H^{-1}(\Omega)\to W_\TT\subseteq H^{-1}(\Omega)$ denote a bounded projection.
  Given $f\in H^{-1}(\Omega)$, let $(u,\ssigma)\in H_0^1(\Omega)\times L^2(\Omega;\R^n)$ denote the solution of~\eqref{eq:fo}.
  The unique solution $(u_\TT,\ssigma_\TT)\in U_\TT\times \Sigma_\TT$ of~\eqref{eq:fosls:mod} satisfies
  \begin{align*}
    \norm{u-u_\TT}{1} + \norm{\ssigma-\ssigma_\TT}{} \lesssim \min_{(v,\ttau)\in U_\TT\times \Sigma_\TT} \norm{u-v}{1} + \norm{\ssigma-\ttau}{}.
  \end{align*}
\end{theorem}
\begin{proof}
We use the notation from the proof of Theorem~\eqref{thm:quasiopt:mixed}. Let $\widetilde u\in H_0^1(\Omega)$ denote the weak solution of $\Delta\widetilde u = -Q_\TT^\star f$. 
With the triangle inequality, $\norm{u-\widetilde u}1=\norm{(1-Q_\TT^\star)f}{-1}$, $\ssigma=\nabla u$, $\widetilde\ssigma = \nabla\widetilde u$, and the quasi-optimality of the FOSLS in the canonical norms we get
\begin{align*}
  \norm{u-u_\TT}{1} + \norm{\ssigma-\ssigma_\TT}{} &\lesssim \norm{(1-Q_\TT^\star )f}{-1} + \min_{(v,\ttau)\in U_\TT\times\Sigma_\TT}
  \norm{\widetilde u-v}1 + \norm{\widetilde\ssigma-\ttau}{} + \norm{Q_\TT^\star f+\div\ttau}{}.
\end{align*}
We argue as in the proof of Theorem~\ref{thm:quasiopt:mixed} and obtain
\begin{align*}
  \min_{\ttau\in \Sigma_\TT} \norm{\widetilde\ssigma-\ttau}{} + \norm{Q_\TT^\star f+\div\ttau}{} 
  \lesssim \min_{\ttau\in\Sigma_\TT} \norm{\nabla\widetilde u-\ttau}{} \leq \norm{u-\widetilde u}1 + \min_{\ttau\in\Sigma_\TT} \norm{\nabla\widetilde u-\ttau}{}.
\end{align*}
Using $\norm{\widetilde u-v}1\leq \norm{u-v}1 + \norm{u-\widetilde u}1 = \norm{u-v}1 + \norm{(1-Q_\TT^\star)f}{-1}$ and putting all the estimates together we infer
\begin{align*}
  \norm{u-u_\TT}{1} + \norm{\ssigma-\ssigma_\TT}{} &\lesssim  \norm{(1-Q_\TT^\star)f}{-1} + \min_{(v,\ttau)\in U_\TT\times\Sigma_\TT} \norm{u-v}1 + \norm{\ssigma-\ttau}{}
\end{align*}
and the proof is finished with an application of Lemma~\ref{lem:aux}.
\end{proof}

Convergence rates in terms of powers of the maximum mesh-size $h$ for the modified FOSLS~\eqref{eq:fosls:mod} have already been proved in~\cite[Theorem~15]{MINRESsingular}. For completeness we recall the result.
\begin{corollary}\label{cor:quasiopt:fosls}
  Under the assumptions of Theorem~\ref{thm:quasiopt:fosls} suppose additionally that $f\in H^{-1+s}(\Omega)$ for some $s\in[0,1]$.
  Then, 
  \begin{align*}
    \norm{u-u_\TT}{1} + \norm{\ssigma-\ssigma_\TT{}}{} \lesssim h^{\min\{s,s_\Omega\}} \norm{f}{-1+\min\{s,s_\Omega\}}.
  \end{align*}
\end{corollary}

\section{Modified Cl\'ement quasi-interpolator}\label{sec:clement}
Define the Cl\'ement quasi-interpolator by
\begin{align*}
  \JCL_\TT v = \sum_{z\in\VV_0} \ip{v}{\chi_z} \eta_z
\end{align*}
with $0$-th order moments, i.e., $\chi_z \in \PP^0(\TT)$ with $\supp(\chi_z)\subseteq \overline{\Patch_z}$ and
\begin{align*}
  \chi_z|_{\Patch_z} = \frac{1}{|\Patch_z|}.
\end{align*}
This Cl\'ement quasi-interpolator has first-order approximation properties, i.e., 
\begin{align*}
  \norm{v-\JCL_\TT v}{T} \lesssim h_T \norm{\nabla v}{\Patch_T}.
\end{align*}
It can be seen by noting that the operator reproduces constants on the patch $\patch_z$ at vertex $z\in\VV_0$, see, e.g.,~\cite{CarstensenClement}.
However, the operator $\JCL_\TT$, in general, does not have second-order approximation properties. 
Below we define a weighted Cl\'ement quasi-interpolator with second-order approximation properties.

There is a simple relation between $\JSZ_\TT$ and $\JCL_\TT$, namely, 
\begin{align*}
  \JSZ_\TT\Pi_\TT^0 = \JCL_\TT,
\end{align*}
see,~\cite[Lemma~21]{MINRESsingular}. Together with $\Bop_\TT\Pi_\TT^0=\Bop_\TT$ (which follows from the definition of $\Bop_\TT$) one sees that
\begin{align*}
  \QCL_\TT = \Pi_\TT^0\big(\JSZ_\TT+\Bop_\TT(1-\JSZ_\TT)\big)' = (\JCL_\TT)' + (1-\JCL_\TT)'\Bop_\TT',
\end{align*}
where
\begin{align*}
  (\JCL_\TT)'\phi = \sum_{z\in\VV_0} \ip{\phi}{\eta_z} \chi_z.
\end{align*}

If the mesh satisfies a certain symmetry condition, then it can be shown that $\JCL_\TT$ also has second-order approximation properties
although its argument is averaged over a nodal patch.
To that end, given $T\in\TT$ let $s_T = (n+1)^{-1}\sum_{z\in\VV_T}z$ denote its center of mass.
For any $z\in\VV_0$ the centroid of its patch $\Patch_z$ is given by
\begin{align*}
  s_z = \frac{1}{|\Patch_z|}\sum_{T\in\patch_z} |T|\,s_T.
\end{align*}
The following result is found in~\cite[Lemma~22]{MINRESsingular}:
\begin{proposition}
  If $s_z=z$ for all $z\in\VV_0$, then
  \begin{align*}
    \norm{(1-\JCL_\TT)v}{} \lesssim h^2\norm{D^2 v}{} \quad\text{for all } v\in H^2(\Omega)\cap H_0^1(\Omega). 
  \end{align*}
\end{proposition}
The latter result is based on the observation that if $z=s_z$, then $\JCL_\TT q(z) = q(z)$ for $q$ a polynomial of degree less or equal to one. 
This property is lost when $z\neq s_z$, see Section~\ref{sec:ex:weightedClement} for a numerical example.
In particular, (following the proof of~\cite[Lemma~22]{MINRESsingular}):
\begin{align*}
  (\JCL_\TT q)(z) = \frac{1}{|\Patch_z|} \int_{\Patch_z} q \,\di x = \frac{1}{|\Patch_z|} \sum_{T\in\patch_z} \int_T q\,\di x 
  = \frac{1}{|\Patch_z|} \sum_{T\in\patch_z} |T| \, q(s_T) = q(s_z).
\end{align*}

For the construction of the weighted Cl\'ement quasi-interpolator consider for each $z\in\VV_0$, $(\alpha_{z,T})_{T\in\patch_z}\subset \R$ such that
\begin{align}\label{eq:comb}
  \sum_{T\in\patch_z} \alpha_{z,T}  s_T = z, \quad\sum_{T\in\patch_z} \alpha_{z,T} = 1, 
  \quad \max_{T\in\patch_z} |\alpha_{z,T}| \lesssim 1.
\end{align}
We stress that such a combination always exists as can be seen as follows. Let $y\in\R^n$ be given and $f_j(x) = x_j$ for $x=(x_1,\ldots,x_n)\in\R^n$, $j=1,\ldots,n$. Then, for each $j\in\{1,\ldots,n\}$, $\div(f_jy) = y_j$ and integration by parts gives
\begin{align*}
  \frac{y_j|\Patch_z|}{n+1}=\sum_{T\in\patch_z} \frac{y_j|T|}{n+1} = \int_{\Patch_z} \div(f_jy)\eta_z \,\di x
  =-\int_{\Patch_z} x_j(y\cdot\nabla\eta_z)\,\di x = -\sum_{T\in\patch_z} |T|(y\cdot\nabla\eta_z|_T)(s_T)_j.
\end{align*}
Therefore, $y = -\tfrac{(n+1)}{|\Patch_z|}\sum_{T\in\patch_z} |T|(y\cdot\nabla\eta_z|_T) s_T$ which proves that the centers of mass $\set{s_T}{T\in\patch_z}$ span $\R^n$. 
Since $\sum_{T\in\patch_z}|T|\nabla\eta_z|_T = \int_{\Patch_z}\nabla\eta_z \,\di x = 0$ the coefficients sum up to $0$. 
To obtain coefficients that sum up to $1$ we proceed as follows. 
Using the established identity for the centroid $y=s_z$ gives
\begin{align*}
  0 = s_z-s_z = \sum_{T\in\patch_z}\frac{|T|}{|\Patch_z|}(1+(n+1)s_z\cdot\nabla\eta_z|_T)s_T.
\end{align*}
Adding the last identity to the formula for $y$ then proves that for all $y\in\R^n$
\begin{align*}
  y = \sum_{T\in\patch_z} \frac{|T|}{|\Patch_z|} (1+(n+1)(s_z-y)\cdot\nabla\eta_z|_T)s_T =: \sum_{T\in\patch_z} \beta_{z,T}(y)s_T
\end{align*}
and $\sum_{T\in\patch_z}\beta_{z,T}(y) =  1$. 
Note that $|\beta_{z,T}(z)|\lesssim 1+(n+1)|s_z-z|\,|\nabla\eta_z|_T| \lesssim 1$ for $T\in\patch_z$. Consequently, setting $\alpha_{z,T} = \beta_{z,T}(z)$, $T\in\patch_z$ satisfies~\eqref{eq:comb}.

The combination in~\eqref{eq:comb} is not necessarily unique for $n\geq 2$.
For certain patches the coefficients $\alpha_{z,T}$ can be chosen to be non-negative so that $z$ can be written as a convex combination of the centers of mass. A sufficient (but not necessary) condition is $\beta_{z,T}(z)\geq 0$ for all $T\in\patch_z$, in which case one can take $\alpha_{z,T}=\beta_{z,T}(z)$.
Further comments on this matter are given in the Appendix~\eqref{sec:convexcomb}.

\begin{example}\label{ex:CL1d}
  Let $\Omega = (a,b)$ and let $\TT$ denote a partition of $\Omega$ into open intervals. 
  For an interior node $z\in\VV_0$ let $T_z^- = (z_-,z)$ and $T_z^+=(z,z_+)$ denote the two elements of the patch $\patch_z$. 
  A straightforward computation shows that $\alpha_{z,T_z^-},\alpha_{z,T_z^+}$ satisfying~\eqref{eq:comb} are unique and given by
  \begin{align*}
    \alpha_{z,T_z^-} = \frac{z_+-z}{z_+-z_-}, \quad \alpha_{z,T_z^+} = \frac{z-z_-}{z_+-z_-}.
  \end{align*}
  Note that both coefficients are non-negative, i.e., $\alpha_{z,T_z^{\pm}}\geq 0$.
\end{example}
\begin{example}\label{ex:CL2d}
  Consider $n=2$ and the nodes $z_1 = (0,0)$, $z_2 = (1,0)$, $z_3 = (1,1)$, $z_4 = (0,1)$ and $z=(\tfrac12,\tfrac13)$. The elements $T_j = \conv\{z_j,z_{\!\!\mod(j,4)+1},z\}$, $j=1,2,3,4$, define a regular triangulation of the domain $\Omega = (0,1)^2$. The centers of mass are given by
  \begin{align*}
    s_{T_1} = (\tfrac12,\tfrac19), \quad s_{T_2} = (\tfrac56,\tfrac49), \quad
    s_{T_3} = (\tfrac12,\tfrac79), \quad s_{T_4} = (\tfrac16,\tfrac49).
  \end{align*}
  It can be verified that~\eqref{eq:comb} does not uniquely determine the coefficients, e.g., 
  \begin{align*}
    \frac23 s_{T_1} + \frac13 s_{T_3} = z = \frac13 s_{T_1} + \frac13 s_{T_2} + \frac13 s_{T_4},
  \end{align*}
  as well as $\beta_{z,T_1}(z) = \tfrac5{12}$, $\beta_{z,T_2}(z)= \tfrac14$, $\beta_{z,T_3}(z)=\tfrac1{12}$, $\beta_{z,T_4}(z) = \tfrac14$.
  Note that for all these combinations the coefficients are non-negative and are therefore convex combinations. 
\end{example}
For each $z\in\VV_0$ let $(\alpha_{z,T})_{T\in\patch_z}$ denote fixed coefficients satisfying~\eqref{eq:comb} and define
\begin{align*}
  \varphi_z|_T = \begin{cases}
    \frac{\alpha_{z,T}}{|T|} & \text{if } T\in\patch_z, \\
    0 & \text{else}.
  \end{cases}
\end{align*}
Thus, $\varphi_z\in \PP^0(\TT)$ and $\norm{\varphi_z}{L^\infty(\Omega)}\lesssim |\Patch_z|^{-1}$ by shape-regularity.  

Let $\weight = \set{\varphi_z}{z\in\VV_0}$ denote the collection of all weight functions. 
The weighted Cl\'ement quasi-interpolator is given by
\begin{align}
  \JCLw_\TT v = \sum_{z\in\VV_0} \ip{v}{\varphi_z} \eta_z \quad\text{for }v\in L^1(\Omega).
\end{align}
We collect its main properties in the next result.
\begin{theorem}\label{thm:wClem}
  The weighted Cl\'ement quasi-interpolator satisfies:
  \begin{itemize}
    \item $\JCLw_\TT v= \JCLw_\TT\Pi_\TT^0v$ for all $v\in L^2(\Omega)$, 
    \item $\norm{(1-\JCLw_\TT)v}T\lesssim h_T \norm{\nabla v}{\Patch_T}$ for all $T\in\TT$ and $v\in H_0^1(\Omega)$, 
    \item $\norm{(1-\JCLw_\TT)v}T\lesssim h_T^2 \norm{D^2 v}{\Patch_T}$ for all $T\in\TT$ and $v\in H^2(\Omega)\cap H_0^1(\Omega)$,
    \item $\norm{\JCLw_\TT v}{}\lesssim \norm{v}{}$ for all $v\in L^2(\Omega)$, resp. $\norm{\JCLw_\TT v}{1}\lesssim \norm{v}1$ for all $v\in H_0^1(\Omega)$. 
  \end{itemize}
\end{theorem}
\begin{proof}
  To see the first assertion note that $\varphi_z\in\PP^0(\TT)$, thus, 
  \begin{align*}
    \JCLw_\TT v = \sum_{z\in\VV_0} \ip{v}{\varphi_z}\eta_z
    = \sum_{z\in\VV_0} \ip{v}{\Pi_\TT^0\varphi_z}\eta_z = \sum_{z\in\VV_0} \ip{\Pi_\TT^0 v}{\varphi_z}\eta_z
    = \JCLw_\TT \Pi_\TT^0 v.
  \end{align*}
  Boundedness in $L^2(\Omega)$ follows from local boundedness. Let $T\in\TT$ be given. 
    We get with the usual scaling arguments together with $\norm{\varphi_z}{L^\infty(\Omega)}\lesssim |\Patch_z|^{-1}$ that
  \begin{align*}
    \norm{\JCLw_\TT v}T \leq \sum_{z\in\VV_T\cap\VV_0} |\ip{v}{\varphi_z}|\, \norm{\eta_z}{T} \lesssim \sum_{z\in\VV_T\cap\VV_0} \norm{v}{\Patch_z}|\Patch_z|^{-1}\norm{1}{\Patch_z}|T|^{1/2} \lesssim \norm{v}{\Patch_T}.
  \end{align*}
  The first-order and second-order approximation properties can be seen as follows.
  Let $q$ be a polynomial of degree less than or equal to one on $\Patch_z$. Then, for $z\in\VV_0$ we have
  \begin{align*}
    \JCLw_\TT q(z) = \int_{\Patch_z} \varphi_z q\,\di x 
    &=\sum_{T\in\patch_z} \frac{\alpha_{z,T}}{|T|}\int_T q\,\di x = \sum_{T\in\patch_z} \frac{\alpha_{z,T}}{|T|} |T|q(s_T)
    \\ 
    &= \sum_{T\in\patch_z} \alpha_{z,T} \, q(s_T) = q\Big(\sum_{T\in\patch_z}\alpha_{z,T}s_T\Big) = q(z).
  \end{align*}
  Here, we used~\eqref{eq:comb}.
  Let $T\in\TT$ be given and let $q$ denote a polynomial of degree less than or equal to one on $\Patch_T$ with $q(z)=0$ for $z\in \VV_T\setminus\VV_0$. With the aforegoing observations we see that $(\JCLw_\TT q)(z) = q(z)$ for $z\in\VV_T$ and, consequently, $\JCLw_\TT q|_T = q|_T$. 
    This and the local boundedness yield
  \begin{align*}
    \norm{(1-\JCLw_\TT)v}T &= \norm{(1-\JCLw_\TT)(v-q)}T \lesssim \norm{v-q}{\Patch_T}.
  \end{align*}
  The asserted approximation results then follow by a Bramble--Hilbert argument.
  
  It remains to prove boundedness in $H_0^1(\Omega)$. Let $T\in\TT$ be given and let $q$ denote a constant on $\Patch_T$ with $q(z) = 0$ for $z\in\VV_T\setminus\VV_0$. Arguing as above we have $\JCLw_\TT q|_T = q|_T$. 
    The inverse estimate and local boundedness then show
  \begin{align*}
    \norm{\nabla\JCLw_\TT v}T &= \norm{\nabla\JCLw_\TT(v-q)}T \lesssim h_T^{-1}\norm{\JCLw_\TT(v-q)}T\lesssim h_T^{-1}\norm{v-q}{\Patch_T}.
\end{align*}
Again, with a Bramble--Hilbert argument we conclude $h_T^{-1}\norm{v-q}{\Patch_T}\lesssim \norm{\nabla v}{\Patch_T}$, which finishes the proof.
\end{proof}

Following Section~\ref{sec:projections} we define a bounded projector $\QCLw_\TT\colon H^{-1}(\Omega)\to \PP^0(\TT)$ based on the weighted Cl\'ement operator as
\begin{align*}
  \QCLw_\TT := (\JCLw_\TT)' + (1-\JCLw_\TT)'\Bop_\TT'.
\end{align*}

We summarize its properties in the next result. The proof follows similar as in~\cite[Theorem~8]{MultilevelNorms21} and we only give some details:
\begin{theorem}\label{thm:Qprojmod}
  The operator $\QCLw_\TT$ satisfies
  \begin{itemize}
    \item $\QCLw_\TT\phi = \phi$ for all $\phi\in \PP^0(\TT)$, 
    \item $\norm{(1-\QCLw_\TT)\phi}{-1} \lesssim \norm{h_\TT(1-\Pi_\TT^0)\phi}{}$ for all $\phi\in L^2(\Omega)$, 
    \item $\norm{\QCLw_\TT\phi}{-1}\lesssim \norm{\phi}{-1}$ for all $\phi \in H^{-1}(\Omega)$, and
    \item $\norm{\QCLw_\TT\phi}{}\lesssim \norm{\phi}{}$ for all $\phi \in L^2(\Omega)$.
  \end{itemize}
\end{theorem}
\begin{proof}
  Using $J_\TT = (\QCLw_\TT)' = \JCLw_\TT + \Bop_\TT(1-\JCLw_\TT)$, boundedness of $B_\TT$, and 
  Theorem~\ref{thm:wClem} we obtain 
  \begin{align*}
    \norm{J_\TT v}{} &\leq \norm{\JCLw_\TT v}{} + \norm{B_\TT(1-\JCLw_\TT )v}{}
    \lesssim \norm{v}{} + \norm{(1-\JCLw_\TT)v}{} \lesssim \norm{v}{}. 
  \end{align*}
  The same arguments together with the inverse estimate and the approximation property of $\JCLw_\TT$ prove that
  \begin{align*}
    \norm{\nabla J_\TT v}{} \lesssim \norm{\nabla v}{} + \norm{h_\TT^{-1}B_\TT(1-\JCLw_\TT)v}{} \lesssim \norm{\nabla v}{}.
  \end{align*}
  Therefore, $\QCLw_\TT = J_\TT'$ is bounded in $L^2(\Omega)$ resp. $H^{-1}(\Omega)$. 

  The projection property can be seen by noting that
  \begin{align*}
    \ip{\QCLw_\TT\phi}{v} = \ip{\phi}{\JCLw_\TT + \Bop_\TT(1-\JCLw_\TT)v} = 0
  \end{align*}
  for all $\phi\in \PP^0(\TT)$, $v\in L^2(\Omega)$. 

  Finally, the projection property and boundedness of $\QCLw_\TT$ yield
  \begin{align*}
  \norm{(1-\QCLw_\TT)\phi}{-1} = \norm{(1-\QCLw_\TT)(1-\Pi_\TT^0)\phi}{-1} \lesssim \norm{(1-\Pi_\TT^0)\phi}{-1} \lesssim \norm{h_\TT(1-\Pi_\TT^0)\phi}{} 
  \end{align*}
  for $\phi\in L^2(\Omega)$, which concludes the proof.
\end{proof}

The next result provides insight into the best-approximation of constants in the dual norm of $H^2(\Omega)\cap H_0^1(\Omega)$.
\begin{corollary}\label{cor:Hm2approx}
  Let $X= H^2(\Omega)\cap H_0^1(\Omega)$. 
  If $f\in L^2(\Omega)$, then
  \begin{align*}
    \min_{f_\TT\in \PP^0(\TT)} \norm{f-f_\TT}{X'} \leq \norm{(1-\QCLw_\TT)f}{X'} \lesssim \norm{h_\TT^2(1-\Pi_\TT^0)f}{}.
  \end{align*}
\end{corollary}
\begin{proof}
  Let $v\in H^2(\Omega)\cap H_0^1(\Omega)$. Then, 
  \begin{align*}
    |\ip{(1-\QCLw_\TT)f}{v}| &= |\ip{(1-\QCLw_\TT)(f-f_\TT)}v| 
    \\
    &= |\ip{f-f_\TT}{(1-\JCLw_\TT)v - \Bop_\TT(1-\JCLw_\TT)v}|
    \\
    &\leq \sum_{T\in\TT} \norm{f-f_\TT}T \norm{(1-\Bop_\TT)(1-\JCLw_\TT)v}T 
    \\
    &\lesssim \sum_{T\in\TT} \norm{f-f_\TT}T \norm{(1-\JCLw_\TT)v}T
  \end{align*}
  for all $f_\TT\in \PP^0(\TT)$. Here, we also used boundedness of $\Bop_\TT|_T\colon L^2(T)\to L^2(T)$. 
  Applying Theorem~\ref{thm:wClem} shows $\norm{(1-\JCLw_\TT)v}{T}\lesssim h_T^2\norm{D^2 v}{\Patch_T}$.
  We conclude that
  \begin{align*}
    \norm{(1-\QCLw_\TT)f}{X'} &= \sup_{v\in X\setminus\{0\}} \frac{\ip{(1-\QCLw_\TT)f}{v}}{\norm{v}{\Omega,2}} \lesssim \norm{h_\TT^2(1-\Pi_\TT^0)f}{}.
  \end{align*}
  This finishes the proof.
\end{proof}

\section{$L^2$ estimates and postprocessed solution}\label{sec:L2}
In this section we revisit a well-known postprocessing scheme for mixed methods, see~\cite{Stenberg91}.
We show that using the operator $Q_\TT^\star = \QCLw_\TT$ in~\eqref{eq:mixed:mod} yields an improved result on the convergence of a postprocessed solution in the lowest-order case. 
It is known that the accuracy of postprocessed solutions $u_\TT^\star$ hinges on a closeness result of the approximate solution, e.g.,~\cite[Remark~2.1]{Stenberg91} notes that 
\begin{align*}
  \norm{\Pi_\TT^0u-u_\TT}{} \lesssim h^2 \norm{u}3,
\end{align*}
where $u_\TT\in\PP^0(\TT)$ is the solution of~\eqref{eq:mixed:standard}, to ensure that $\norm{u-u_\TT^\star}{} = \OO(h^2)$. 
The problem with this estimate is that it requires $u\in H^3(\Omega)$ or at least $u|_T\in H^3(T)$, $T\in\TT$, which for $f\in L^2(T)\setminus H^1(T)$, $T\in\TT$, is not realistic because even for the simplest model problem as considered in this work we can not expect more than $H^2(\Omega)$ regularity on convex domains. 

For the analysis in this section we will use the solution $\widetilde u\in H_0^1(\Omega)$ of the auxiliary problem
\begin{align}\label{eq:aux}
  -\Delta \widetilde u = \QCLw_\TT f, \quad\text{and set } \widetilde\ssigma = \nabla\widetilde u.
\end{align}

\subsection{Improved convergence of postprocessed solution}
We start by analyzing a supercloseness property.
\begin{lemma}\label{lem:superclose}
  Let $f\in H^{-1+s}(\Omega)$ for some $s\in[0,1]$ and let $(u_\TT,\ssigma_\TT)\in W_\TT\times \Sigma_\TT$ denote the solution of~\eqref{eq:mixed:mod} with $Q_\TT^\star = \QCLw_\TT$.
  The estimate
  \begin{align*}
  \norm{\Pi_\TT^0 \widetilde u-u_\TT}{} \lesssim h^{s_\Omega}\norm{\widetilde\ssigma-\ssigma_\TT}{} \lesssim h^{s_\Omega+\min\{s,s_\Omega\}}\norm{f}{-1+\min\{s,s_\Omega\}}
  \end{align*}
  holds, where $\widetilde u$ denotes the solution of~\eqref{eq:aux} and $\widetilde\ssigma = \nabla\widetilde u$.
\end{lemma}
\begin{proof}
  The arguments used are essentially the same as in~\cite[Theorem~2.1]{Stenberg91} using the auxiliary solution $\widetilde u$ instead of $u$. 
  Let $(v,\ttau)\in H_0^1(\Omega)\times \Hdivset\Omega$ denote the unique solution of the first-order system,
  \begin{align*}
    \div\ttau &= \Pi_\TT^0 \widetilde u-u_\TT, \\
    \ttau-\nabla v &= 0, \\
    v|_\Gamma &= 0.
  \end{align*}
  Elliptic regularity~\eqref{eq:ellipticreg} shows
  \begin{align*}
    \norm{v}{1+s_\Omega} + \norm{\ttau}{s_\Omega} \lesssim \norm{\Pi_\TT^0 \widetilde u-u_\TT}{-1+s_\Omega} \lesssim \norm{\Pi_\TT^0 \widetilde u-u_\TT}{}.
  \end{align*}
  We have together with integration by parts that
  \begin{align*}
    \norm{\Pi_\TT^0 \widetilde u-u_\TT}{}^2 &= \ip{\widetilde u-u_\TT}{\Pi_\TT^0 \widetilde u-u_\TT} = \ip{\widetilde u-u_\TT}{\div\ttau}
    \\
    &= \ip{\widetilde u-u_\TT}{\div\ttau} + \ip{\widetilde\ssigma-\ssigma_\TT}{\ttau-\nabla v} 
    \\
    &= \ip{\widetilde u-u_\TT}{\div\ttau} + \ip{\widetilde\ssigma-\ssigma_\TT}{\ttau} + \ip{\div(\widetilde\ssigma-\ssigma_\TT)}{v} 
    \\
    &= \ip{\widetilde u-u_\TT}{\div(\ttau-\ttau_\TT)} + \ip{\widetilde\ssigma-\ssigma_\TT}{\ttau-\ttau_\TT} + \ip{\div(\widetilde\ssigma-\ssigma_\TT)}{v}
  \end{align*}
  for all $\ttau_\TT\in\RT^0(\TT)$. The last identity follows from Galerkin orthogonality. 
  Choosing $\ttau_\TT= \Pi_\TT^{\RT}\ttau$ we get that $\div(\ttau-\ttau_\TT) = (1-\Pi_\TT^0)\div\ttau = 0$ by~\eqref{eq:propRTproj}
  and noting that by~\eqref{eq:mixed:mod} the equality $\div\widetilde\ssigma = -\QCLw_\TT f= \div\ssigma_\TT$ holds, we further infer that
  \begin{align*}
    \norm{\Pi_\TT^0 \widetilde u-u_\TT}{}^2 &= \ip{\widetilde u-u_\TT}{\div(\ttau-\ttau_\TT)} + \ip{\widetilde\ssigma-\ssigma_\TT}{\ttau-\ttau_\TT} + \ip{\div(\widetilde\ssigma-\ssigma_\TT)}{v}
    \\
    &=  \ip{\widetilde\ssigma-\ssigma_\TT}{\ttau-\ttau_\TT} \leq \norm{\widetilde\ssigma-\ssigma_\TT}{}\norm{\ttau-\Pi_\TT^{\RT}\ttau}{} 
    \\
    &\lesssim \norm{\widetilde\ssigma-\ssigma_\TT}{} \left( \sum_{T\in\TT}\min_{\cchi\in \RT^0(T)} \norm{\ttau-\cchi}T^2 \right)^{1/2}
    \lesssim h^{s_\Omega} \norm{\widetilde\ssigma-\ssigma_\TT}{}  \norm{\Pi_\TT^0 \widetilde u-u_\TT}{}.
  \end{align*}
  The last estimate is a consequence of the fact that $\PP^0(T)^n\subseteq \RT^0(T)$, the approximation property of piecewise constants and elliptic regularity.

  It remains to prove $\norm{\widetilde\ssigma-\ssigma_\TT}{} \lesssim h^{\min\{s,s_\Omega\}}\norm{f}{-1+\min\{s,s_\Omega\}}$:
  The triangle inequality together with stability of the Poisson problem and quasi-optimality (Corollary~\ref{cor:quasiopt:mixed}) implies that
  \begin{align*}
    \norm{\widetilde\ssigma-\ssigma_\TT}{} &\leq \norm{\nabla u-\nabla\widetilde u}{} + \norm{\ssigma-\ssigma_\TT}{} 
    \\
    &\lesssim \norm{(1-\QCLw_\TT)f}{-1} + h^{\min\{s,s_\Omega\}} \norm{f}{-1+\min\{s,s_\Omega\}}.
  \end{align*}
  From Theorem~\ref{thm:Qprojmod} we conclude $\norm{(1-\QCLw_\TT)f}{-1} \lesssim h^{\min\{s,s_\Omega\}} \norm{f}{-1+\min\{s,s_\Omega\}}$ with an interpolation argument. This finishes the proof.
\end{proof}

We investigate the following postprocessing scheme, see, e.g.~\cite{Stenberg91} for mixed schemes or~\cite{SupConv2} for the discontinuous Petrov--Galerkin method with optimal test functions. 
Let $(u_\TT,\ssigma_\TT)\in W_\TT\times \Sigma_\TT$ denote the solution of~\eqref{eq:mixed:mod}. Define $u_\TT^\star\in \PP^1(\TT)$ on each $T\in\TT$ by
\begin{subequations}\label{eq:postproc}
  \begin{align}
    \ip{\nabla u_\TT^\star}{\nabla v}_T &= \ip{\ssigma_\TT}{\nabla v}_T \quad\forall v\in \PP^1(T), \\
    \Pi_\TT^0 u_\TT^\star|_T &= u_\TT|_T.
  \end{align}
\end{subequations}
Note that the postprocessing scheme from~\cite[Eq.(2.16)]{Stenberg91} is, in general, not well defined if $f\in H^{-1}(\Omega)\setminus L^2(\Omega)$.
Replacing the load $f$ with $\QCLw_\TT f$ in~\cite[Eq.(2.16)]{Stenberg91} and using that $\div\ssigma_\TT = -\QCLw_\TT f$ for the solution of~\eqref{eq:mixed:mod} we get~\eqref{eq:postproc} after integrating by parts.

\begin{theorem}\label{thm:postproc}
  Let $f\in H^{-1+s}(\Omega)$ for some $s\in[0,1]$ and let $(u_\TT,\ssigma_\TT)\in W_\TT\times \Sigma_\TT$ denote the solution of~\eqref{eq:mixed:mod} with $Q_\TT^\star = \QCLw_\TT$.
  We have that
  \begin{align*}
    \norm{u-u_\TT^\star}{} \lesssim h^{s_\Omega+\min\{s,s_\Omega\}} \norm{f}{-1+\min\{s,s_\Omega\}}.
  \end{align*}
  In particular, if $\Omega$ is convex and $f\in L^2(\Omega)$, then
  \begin{align*}
    \norm{u-u_\TT^\star}{} \lesssim h^2 \norm{f}{}.
  \end{align*}
\end{theorem}
\begin{proof}
  Using the triangle inequality and $\Pi_\TT^0u_\TT^\star = u_\TT = \Pi_\TT^0 u_\TT$ we get 
  \begin{align}\label{eq:proof:postproc}
  \begin{split}
    \norm{u-u_\TT^\star}{} &\leq \norm{u-\widetilde u}{} + \norm{\widetilde u-u_\TT^\star}{} 
    \\
    &\leq \norm{u-\widetilde u}{} + \norm{(1-\Pi_\TT^0)(\widetilde u-u_\TT^\star)}{} + \norm{\Pi_\TT^0(\widetilde u- u_\TT^\star)}{} 
    \\
    &\lesssim \norm{u-\widetilde u}{} + \left(\sum_{T\in\TT} h_T^2\norm{\nabla(\widetilde u- u_\TT^\star)}T^2\right)^{1/2} + 
    \norm{\Pi_\TT^0\widetilde u- u_\TT}{}.
  \end{split}
  \end{align}
  The last term on the right-hand side is estimated with Lemma~\ref{lem:superclose}, 
  \begin{align*}
    \norm{\Pi_\TT^0\widetilde u- u_\TT}{} \lesssim h^{s_\Omega+\min\{s,s_\Omega\}} \norm{f}{-1+\min\{s,s_\Omega\}}.
  \end{align*}
  For the first term on the right-hand side of~\eqref{eq:proof:postproc} let $v\in H_0^1(\Omega)$ denote the solution of $-\Delta v = u-\widetilde u$. Then, with elliptic regularity and the properties of $\QCLw_\TT$ resp. $(\QCLw_\TT)'$ we infer that 
  \begin{align*}
    \norm{u-\widetilde u}{}^2 &= \ip{u-\widetilde u}{-\Delta v} = \ip{(1-\QCLw_\TT)f}{v} = \ip{(1-\QCLw_\TT)f}{(1-(\QCLw_\TT)')v} \\
    &\lesssim \norm{(1-\QCLw_\TT)f}{-1} \norm{(1-(\QCLw_\TT)')v}{1} \lesssim h^{\min\{s,s_\Omega\}}\norm{f}{-1+\min\{s,s_\Omega\}} h^{s_\Omega} \norm{v}{1+s_\Omega} 
    \\
    &\lesssim h^{s_\Omega+\min\{s,s_\Omega\}}\norm{f}{-1+\min\{s,s_\Omega\}}\norm{u-\widetilde u}{}.
  \end{align*}
  The remaining term on the right-hand side of~\eqref{eq:proof:postproc} is estimated following arguments similar as, e.g., found in~\cite[Proof of Theorem~2.2]{Stenberg91} or~\cite[Section~3.5]{SupConv2}. For the sake of completeness we repeat the steps here. 
  Let $\widetilde u_\TT\in\PP^1(\TT)$ denote the solution of the local Neumann problems
  \begin{align*}
    \ip{\nabla\widetilde u_\TT}{\nabla v}_T &= \ip{\widetilde\ssigma}{\nabla v}_T \quad\forall v\in \PP^1(T), \\
    \ip{\widetilde u_T}1_T &= 0
  \end{align*}
  for each $T\in\TT$. 
  Note that $\widetilde u_\TT$ is the (elementwise) Galerkin approximation of $\widetilde u$. Therefore, 
  \begin{align*}
    \norm{\nabla(\widetilde u-u_\TT^\star)}T &\leq \norm{\nabla(\widetilde u-\widetilde u_\TT)}T + \norm{\nabla(\widetilde u_\TT-u_\TT^\star)}T
    \\
    &\leq \norm{\nabla(\widetilde u-v)}T + \norm{\widetilde\ssigma-\ssigma_\TT}{T}
  \end{align*}
  for all $v\in \PP^1(T)$. Multiplying by $h_T$, summing over all elements, and using the triangle inequality we conclude that
  \begin{align}\label{eq:postproc2}
    \left(\sum_{T\in\TT} h_T^2\norm{\nabla(\widetilde u- u_\TT^\star)}T^2\right)^{1/2} &\lesssim h \norm{\nabla(u-\widetilde u)}{} + h \min_{v\in U_\TT}\norm{u-v}1 + h\norm{\ssigma-\ssigma_\TT}{}.
  \end{align}
  The first term on the right-hand side of~\eqref{eq:postproc2} is estimated using stability of the Poisson problem and properties of $\QCLw_\TT$ leading to
  \begin{align*}
    h\norm{\nabla(u-\widetilde u)}{} \eqsim h\norm{(1-\QCLw_\TT)f}{-1} \lesssim h^{1+\min\{s,s_\Omega\}} \norm{f}{-1+\min\{s,s_\Omega\}}.
  \end{align*}
  For the remaining terms on the right-hand side of~\eqref{eq:postproc2} we use Corollary~\ref{cor:quasiopt:mixed} and Theorem~\ref{thm:quasiopt:fosls} together with Corollary~\ref{cor:quasiopt:fosls} to conclude that
  \begin{align*}
    h \min_{v\in U_\TT}\norm{u-v}1 + h\norm{\ssigma-\ssigma_\TT}{} \lesssim h^{1+\min\{s,s_\Omega\}}\norm{f}{-1+\min\{s,s_\Omega\}}.
  \end{align*}
  Combining all estimates and using $h\lesssim h^{s_\Omega}$ finishes the proof.
\end{proof}

Another consequence of Lemma~\ref{lem:superclose} is the following result.
\begin{corollary}\label{cor:uL2close}
  Let $f\in H^{-1+s}(\Omega)$ for some $s\in[0,1]$ and let $(u_\TT,\ssigma_\TT)\in W_\TT\times \Sigma_\TT$ denote the solution of~\eqref{eq:mixed:mod} with $Q_\TT^\star = \QCLw_\TT$.
  We have that
  \begin{align*}
    \norm{u-u_\TT}{} \leq \norm{u-\Pi_\TT^0u}{} + Ch^{s_\Omega+\min\{s,s_\Omega\}}\norm{f}{-1+\min\{s,s_\Omega\}},
  \end{align*}
  where $C>0$ denotes a generic constant.
\end{corollary}
\begin{proof}
  By the triangle inequality, 
  \begin{align*}
    \norm{u-u_\TT}{} &\leq \norm{u-\Pi_\TT^0u}{} + \norm{\Pi_\TT^0(u-\widetilde u)}{} + \norm{\Pi_\TT^0\widetilde u-u_\TT}{}
    \\
    &\leq \norm{u-\Pi_\TT^0u}{} + \norm{u-\widetilde u}{} + \norm{\Pi_\TT^0\widetilde u-u_\TT}{}.
  \end{align*}
  The last term is estimated with Lemma~\ref{lem:superclose} and the middle term is estimated as in the proof of Theorem~\ref{thm:postproc}.
\end{proof}

\subsection{$L^2$ convergence of FOSLS}
In this section we study the $L^2(\Omega)$ error  $\norm{u-u_\TT}{}$, 
where $(u_\TT,\ssigma_\TT)\in U_\TT\times \Sigma_\TT$ is the solution to the modified FOSLS~\eqref{eq:fosls:mod}. 
For convex domains and $f\in L^2(\Omega)$ we already studied $L^2(\Omega)$ convergence rates in~\cite[Theorem~18]{MINRESsingular} when using the operator $Q_\TT^\star=\QCL_\TT$. 
The following theorem extends the findings of~\cite[Section~4]{MINRESsingular} for $Q_\TT^\star = \QCLw_\TT$ and non-convex domains. 
Its proof follows the same ideas as in~\cite[Section~4]{MINRESsingular} but for sake of completeness we repeat the main arguments here.
Related works on $L^2(\Omega)$ error estimates for the FOSLS include~\cite{CaiKu2010,Ku2011}.
It is important to note that optimal rates, i.e., $\norm{u-u_\TT}{} = \OO(h^2)$ can not be expected for solutions of~\eqref{eq:fosls:standard} even if $f\in L^2(\Omega)$. 
Indeed, we presented a numerical experiment in~\cite[Section~5]{MINRESsingular} that confirms this.

\begin{theorem}\label{thm:L2estimate:fosls}
  Let $f\in H^{-1+s}(\Omega)$ for some $s\in[0,1]$ and let $(u_\TT,\ssigma_\TT)\in U_\TT\times \Sigma_\TT$ denote the solution of~\eqref{eq:fosls:mod} with $Q_\TT^\star = \QCLw_\TT$.
  We have that
  \begin{align*}
    \norm{u-u_\TT}{} \lesssim h^{s_\Omega+\min\{s,s_\Omega\}} \norm{f}{-1+\min\{s,s_\Omega\}}.
  \end{align*}
  In particular, if $\Omega$ is convex and $f\in L^2(\Omega)$, then
  \begin{align*}
    \norm{u-u_\TT}{} \lesssim h^2 \norm{f}{}.
  \end{align*}
\end{theorem}
\begin{proof}
  Considering the splitting $u-u_\TT = u-\widetilde u + \widetilde u-u_\TT$ we get that
  \begin{align*}
    \norm{u-u_\TT}{} \leq \norm{u-\widetilde u}{} + \norm{\widetilde u-u_\TT}{} \lesssim h^{s_\Omega+\min\{s,s_\Omega\}}\norm{f}{-1+\min\{s,s_\Omega\}} + \norm{\widetilde u-u_\TT}{},
  \end{align*}
  where the last estimate follows as in the proof of Theorem~\ref{thm:postproc}.
  Following the proof of~\cite[Theorem~18]{MINRESsingular}, let $w\in H_0^1(\Omega)$ denote the solution of $-\Delta w = \widetilde u-u_\TT$ and let $(v,\ttau)\in H_0^1(\Omega)\times \Hdivset\Omega$ denote the unique solution of the first-order system
  \begin{align*}
    \div\ttau &= -w, \\
    \nabla v-\ttau &= \nabla w.
  \end{align*}
  Then, for any $(v_\TT,\ttau_\TT)\in U_\TT\times\Sigma_\TT$,
  \begin{align*}
    \norm{\widetilde u-u_\TT}{}^2 &= \ip{\widetilde u-u_\TT}{-\Delta w} = \ip{\nabla(\widetilde u-u_\TT)}{\nabla w} 
    \\
    &= \ip{\nabla(\widetilde u-u_\TT)-(\widetilde\ssigma-\ssigma_\TT)}{\nabla w} + \ip{(\widetilde\ssigma-\ssigma_\TT)}{\nabla w}
    \\
    &= \ip{\nabla(\widetilde u-u_\TT)-(\widetilde\ssigma-\ssigma_\TT)}{\nabla v-\ttau} + \ip{\div(\widetilde\ssigma-\ssigma_\TT)}{\div\ttau}
    \\
    &= \ip{\nabla(\widetilde u-u_\TT)-(\widetilde\ssigma-\ssigma_\TT)}{\nabla(v-v_\TT)-(\ttau-\ttau_\TT)} + \ip{\div(\widetilde\ssigma-\ssigma_\TT)}{\div(\ttau-\ttau_\TT)}.
  \end{align*}
  The last identity follows from Galerkin orthogonality (this can be seen by writing down the Euler--Lagrange equations of~\eqref{eq:fosls:mod}).
  Choosing $\ttau_\TT = \Pi_\TT^{\RT}\ttau$ and $v_\TT = \JSZ_\TT v$ we get that $\div(\ttau-\ttau_\TT) = (1-\Pi_\TT^0)(-w)$ and since $\div(\widetilde\ssigma-\ssigma_\TT)\in \PP^0(\TT)$ we also have that 
  \begin{align*}
    \ip{\div(\widetilde\ssigma-\ssigma_\TT)}{\div(\ttau-\ttau_\TT)} = 0.
  \end{align*}
  Using the approximation properties of $\Pi_\TT^{\RT}$ and $\JSZ_\TT$ as well as elliptic regularity we further see that
  \begin{align*}
    \norm{{\nabla(v-v_\TT)-(\ttau-\ttau_\TT)}}{} \lesssim h^{s_\Omega}\norm{v}{1+s_\Omega} + h\norm{(1-\Pi_\TT^0)\div\ttau}{} 
    \lesssim h^{s_\Omega}\norm{\widetilde u-u_\TT}{} + h^2 \norm{\widetilde u-u_\TT}{}.
  \end{align*}
  Here we have used that $\norm{v}{1+s_\Omega}\lesssim \norm{\widetilde u-u_\TT}{}$. This estimate follows from $\Delta v = \Delta w -w$ and elliptic regularity, i.e.,
    \begin{align*}
      \norm{v}{1+s_\Omega}\lesssim \norm{\Delta w}{} + \norm{w}{} \lesssim \norm{\widetilde u-u_\TT}{}.
  \end{align*}
  Putting the above estimates together and using the triangle inequality we infer that
  \begin{align*}
    \norm{\widetilde u-u_\TT}{} &\lesssim h^{s_\Omega}(\norm{\nabla(\widetilde u-u_\TT)}{}+\norm{\widetilde\ssigma-\ssigma_\TT}{})
    \\
    &\lesssim h^{s_\Omega}(\norm{u-u_\TT}{1} + \norm{\ssigma-\ssigma_\TT}{}) + h^{s_\Omega} \norm{u-\widetilde u}{1}
    \\
    &\lesssim h^{s_\Omega}(\norm{u-u_\TT}{1} + \norm{\ssigma-\ssigma_\TT}{}) + h^{s_\Omega} h^{\min\{s,s_\Omega\}}\norm{f}{-1+\min\{s,s_\Omega\}}.
  \end{align*}
  The estimate $\norm{u-\widetilde u}{1} \lesssim h^{\min\{s,s_\Omega\}}\norm{f}{-1+\min\{s,s_\Omega\}}$ has already been used in the proof of Lemma~\ref{lem:superclose}. 
  An application of Corollary~\ref{cor:quasiopt:fosls} finishes the proof.
\end{proof}

\section{Numerical experiments}\label{sec:ex}
We have already studied the FOSLS with $H^{-1}(\Omega)$ loads in our recent work~\cite{MINRESsingular} and we have presented various numerical results in~\cite[Section~5]{MINRESsingular}. 
In Section~\ref{sec:ex:weightedClement} we compare the standard Cl\'ement quasi-interpolator $\JCL_\TT$ to the weighted version $\JCLw_\TT$ for a simple problem in 1D.
Section~\ref{sec:ex:mixed:Hms} deals with a problem where the load is not in $L^2(\Omega)$. 
In Section~\ref{sec:ex:mixed:L2} we consider a problem with $L^2(\Omega)$ load and compare the (postprocessed) solutions of~\eqref{eq:mixed:standard} and~\eqref{eq:mixed:mod}. 
Finally, in Section~\ref{sec:ex:fosls} we compare solutions of the standard FOSLS~\eqref{eq:fosls:standard} and the regularized FOSLS~\eqref{eq:fosls:mod} for a benchmark problem from~\cite{BringmannComp22}.

\subsection{Weighted Cl\'ement operator}\label{sec:ex:weightedClement}
We consider a one-dimensional example and compare the Cl\'ement quasi-interpolator $\JCL_\TT$ to the weighted variant $\JCLw_\TT$. 
To that end let $\Omega = (0,1)$ and $u(x)=\sin(\pi x)$. Clearly, $u\in H^2(\Omega)\cap H_0^1(\Omega)$. 
First, we consider a sequence of meshes where each mesh is a uniform partition of $\Omega$. 
It can be verified by using Example~\ref{ex:CL1d} that $\JCL_\TT = \JCLw_\TT$ and we expect that $\norm{u-\JCL_\TT u}{} = \OO(h^2)$. 
This is confirmed by our computations, see Figure~\ref{fig:1d}. 
Next, we consider a sequence of meshes $\TT_1,\TT_2,\dots$. Each mesh $\TT_j$ is a partition $x_0<x_1<x_2,\dots$ of $\Omega$ such that two adjacent elements have different lengths but the overall mesh is quasi-uniform, i.e.,
\begin{align*}
  x_1-x_0 = h, \quad x_2-x_1 = 2h, \quad x_3-x_2 = h, \quad x_4-x_3 = 2h, \dots
\end{align*}
We expect that 
\begin{align*}
  \norm{u-\JCL_\TT u}{} = \OO(h) \quad\text{and}\quad \norm{u-\JCLw_\TT u}{} = \OO(h^2). 
\end{align*}
This is confirmed by our numerical experiment, see the right plot of Figure~\ref{fig:1d}.

\begin{figure}
  \begin{center}
    \begin{tikzpicture}
\begin{loglogaxis}[
    title={Uniform mesh-size},
width=0.49\textwidth,
cycle list/Dark2-6,
cycle multiindex* list={
mark list*\nextlist
Dark2-6\nextlist},
every axis plot/.append style={ultra thick},
xlabel={number of elements},
grid=major,
legend entries={\small $\|u-\JCL_\TT u\|=\|u-\JCLw_\TT u\|$},
legend pos=south west,
]
\addplot table [x=nE,y=errVL2] {data/ClementIntUnif.dat};
\addplot [black,dotted,mark=none] table [x=nE,y expr={0.6*sqrt(\thisrowno{1})^(1)}] {data/ClementIntUnif.dat};
\addplot [black,dotted,mark=none] table [x=nE,y expr={0.6*sqrt(\thisrowno{1})^(2)}] {data/ClementIntUnif.dat};
\end{loglogaxis}
\end{tikzpicture}
\begin{tikzpicture}
\begin{loglogaxis}[
    title={Non-uniform mesh-size},
width=0.49\textwidth,
cycle list/Dark2-6,
cycle multiindex* list={
mark list*\nextlist
Dark2-6\nextlist},
every axis plot/.append style={ultra thick},
xlabel={number of elements},
grid=major,
legend entries={\small $\|u-\JCL_\TT u\|$,\small $\|u-\JCLw_\TT u\|$},
legend pos=south west,
]
\addplot table [x=nE,y=errVL2] {data/ClementIntNonUnif.dat};
\addplot table [x=nE,y=errVL2w] {data/ClementIntNonUnif.dat};
\addplot [black,dotted,mark=none] table [x=nE,y expr={0.6*sqrt(\thisrowno{1})^(1)}] {data/ClementIntNonUnif.dat};
\addplot [black,dotted,mark=none] table [x=nE,y expr={0.6*sqrt(\thisrowno{1})^(2)}] {data/ClementIntNonUnif.dat};
\end{loglogaxis}
\end{tikzpicture}
  \end{center}
  \caption{$L^2(\Omega)$ errors of the quasi-interpolants $\JCL_\TT u$ and $\JCLw_\TT u$ for $u(x) = \sin(\pi x)$.
    The left plot shows the errors where each mesh is a uniform partition of $\Omega=(0,1)$. 
    The right plot shows the errors where two adjacent elements have different mesh-sizes, but each mesh is quasi-uniform, see Section~\ref{sec:ex:weightedClement}.
    Black dotted lines indicate $\OO(h)$ and $\OO(h^2)$.}\label{fig:1d}
\end{figure}
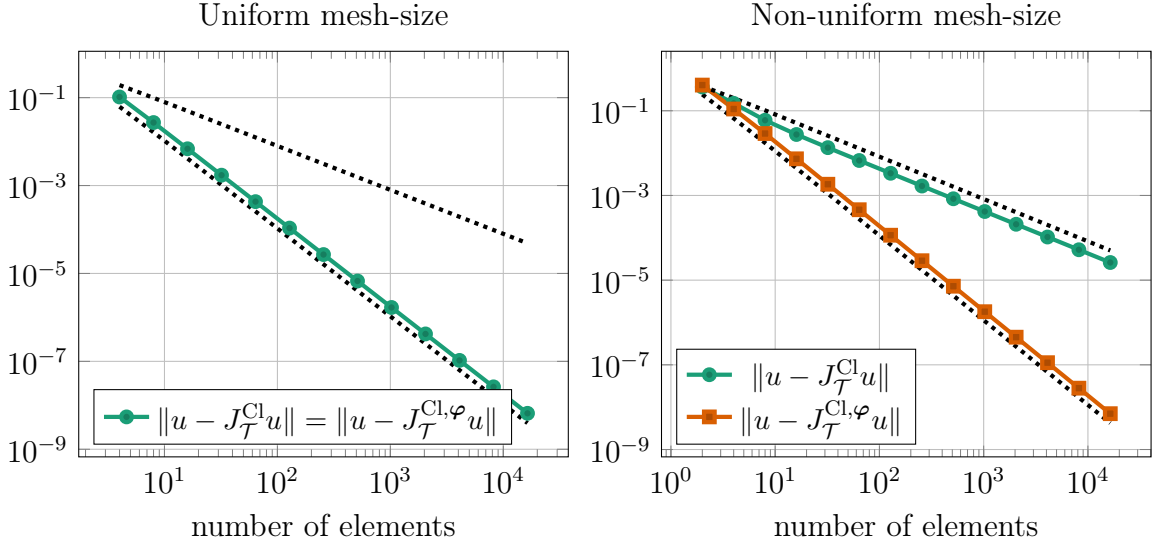

\subsection{Mixed method with $H^{-1}(\Omega)$ load}\label{sec:ex:mixed:Hms}
Let $\Omega = (0,1)^2$ and consider the manufactured solution
\begin{align*}
  u(x,y) = |x-y|^{3/4} \sin(\pi x)\sin(\pi y), \quad (x,y)\in \Omega.
\end{align*}
This solution has also been considered in~\cite[Section~5]{MINRESsingular} for the regularized FOSLS.
We have that $f=-\Delta u \in H^{-1+1/4-\varepsilon}(\Omega)$ for $\varepsilon>0$. 
We study the errors of the solutions $(u_\TT,\ssigma_\TT)$ of the regularized mixed FEM~\eqref{eq:mixed:mod} with $Q_\TT^\star = \QCLw_\TT$.
Recall that $u_\TT^\star\in \PP^1(\TT)$ denotes the postprocessed solution.
The errors are displayed in Figure~\ref{fig:mixed:Hms}. 
It can be observed that
\begin{align*}
  \norm{u-u_\TT}{}=\OO(h), \quad \norm{\ssigma-\ssigma_\TT}{} = \OO(h^{1/4}), \quad
  \norm{u-u_\TT^\star}{} = \OO(h^{1+1/4}),
\end{align*}
in accordance (omitting $\varepsilon$) with the results derived in this work (Corollary~\ref{cor:quasiopt:mixed}, Theorem~\ref{thm:postproc}, and Corollary~\ref{cor:uL2close}).

\begin{figure}
  \begin{center}
    \begin{tikzpicture}
\begin{loglogaxis}[
width=0.49\textwidth,
cycle list/Dark2-6,
cycle multiindex* list={
mark list*\nextlist
Dark2-6\nextlist},
every axis plot/.append style={ultra thick},
xlabel={degrees of freedom},
grid=major,
legend entries={\small $\|u-u_\TT\|$,\small $\|\ssigma-\ssigma_\TT\|$,\small $\|u-u^\star_\TT\|$},
legend pos=south west,
]
\addplot table [x=dofMIXED,y=errUL2] {data/MixedHms.dat};
\addplot table [x=dofMIXED,y=errSigmaL2] {data/MixedHms.dat};
\addplot table [x=dofMIXED,y=errUstar] {data/MixedHms.dat};
\addplot [black,dotted,mark=none] table [x=dofMIXED,y expr={0.6*sqrt(\thisrowno{1})^(-1-1/4)}] {data/MixedHms.dat};
\addplot [black,dotted,mark=none] table [x=dofMIXED,y expr={0.6*sqrt(\thisrowno{1})^(-1)}] {data/MixedHms.dat};
\addplot [black,dotted,mark=none] table [x=dofMIXED,y expr={0.8*sqrt(\thisrowno{1})^(-1/4)}] {data/MixedHms.dat};
\end{loglogaxis}
\end{tikzpicture}
  \end{center}
  \caption{Results for the mixed method for the experiment from Section~\ref{sec:ex:mixed:Hms}. Black dotted lines indicate $\OO(h^{1/4})$, $\OO(h)$ and $\OO(h^{1+1/4})$.}\label{fig:mixed:Hms}
\end{figure}
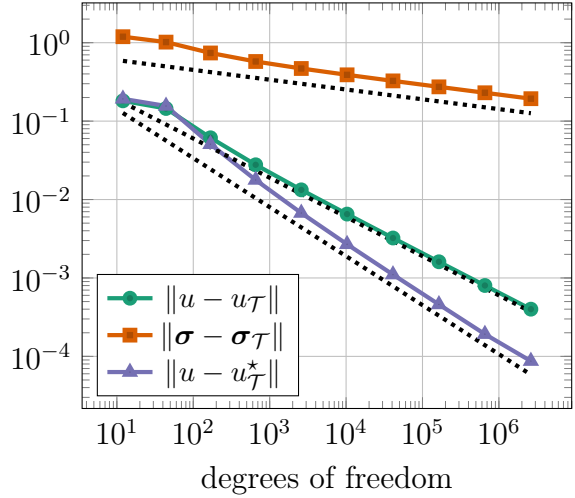

\subsection{Postprocessing in the mixed FEM with $L^2(\Omega)$ load}\label{sec:ex:mixed:L2}
We consider the problem setup of~\cite[Section~6.3]{MINRESsingular}, i.e., the manufactured solution
\begin{align*}
  u(x,y) = x|x|^{1/2+1/128}(1-x^2)(1-y^2), \quad(x,y)\in \Omega:=(-1,1)^2.
\end{align*}
One verifies that $u\in H^2(\Omega)\cap H_0^1(\Omega)$, thus, $f:=-\Delta u \in L^2(\Omega)$ but $f\notin H^s(\Omega)$ for $s\geq \tfrac1{128}$.
In particular, the standard mixed method~\eqref{eq:mixed:standard} and the modified mixed method~\eqref{eq:mixed:mod} with $Q_\TT^\star = \QCLw_\TT$ are well defined. 
Table~\ref{tab:mixed:L2} shows a comparison of the errors: For both methods we observe that
\begin{align*}
  \norm{u-u_\TT}{} + \norm{\ssigma-\ssigma_\TT}{} = \OO(h),
\end{align*}
as expected, whereas for the postprocessed solutions we see that 
\begin{align*}
  \norm{u-u_\TT^\star}{} = 
  \begin{cases}
    \OO(h^{3/2}) & \text{for method~\eqref{eq:mixed:standard}}, \\
    \OO(h^2) &\text{for method~\eqref{eq:mixed:mod}}.
  \end{cases}
\end{align*}
The fact that the postprocessed solution of the modified mixed FEM converges optimally, fits the theory (Theorem~\ref{thm:postproc}). 
We note that, although $f\in L^2(\Omega)$ (so that there would be no need to regularize the datum), the regularized method seems to deliver more accurate solutions. 

\begin{table}[htbp]
  \begin{center}
  {\tiny
  \begin{tabular}{|c|cc|cc|cc|cc|cc|cc|}
    \hline
    \multicolumn{1}{|c|}{} &\multicolumn{6}{|c|}{standard mixed method~\eqref{eq:mixed:standard}} &\multicolumn{6}{|c|}{regularized mixed method~\eqref{eq:mixed:mod}}\\
     \cline{2-13}
     $\#\TT$ & $\norm{\ssigma-\ssigma_\TT}{}$ & eoc & $\norm{u-u_\TT}{}$ & eoc & $\norm{u-u_\TT^\star}{}$ & eoc & $\norm{\ssigma-\ssigma_\TT}{}$ & eoc & $\norm{u-u_\TT}{}$ & eoc & $\norm{u-u_\TT^\star}{}$ & eoc \\\hline
4 & 1.35e+00 & --- & 3.42e-01 & --- & 4.45e-01 & --- & 1.53e+00 & --- & 4.10e-01 & --- & 5.32e-01 & --- \\\hline
16 & 7.95e-01 & 0.77 & 1.57e-01 & 1.12 & 1.46e-01 & 1.61 & 8.04e-01 & 0.92 & 1.65e-01 & 1.31 & 1.56e-01 & 1.77 \\\hline
64 & 4.83e-01 & 0.72 & 9.00e-02 & 0.81 & 5.41e-02 & 1.43 & 4.79e-01 & 0.75 & 9.19e-02 & 0.85 & 5.72e-02 & 1.45 \\\hline
256 & 2.57e-01 & 0.91 & 4.79e-02 & 0.91 & 1.51e-02 & 1.84 & 2.52e-01 & 0.92 & 4.81e-02 & 0.93 & 1.57e-02 & 1.87 \\\hline
1024 & 1.32e-01 & 0.97 & 2.43e-02 & 0.98 & 4.17e-03 & 1.86 & 1.28e-01 & 0.97 & 2.43e-02 & 0.98 & 4.03e-03 & 1.96 \\\hline
4096 & 6.66e-02 & 0.98 & 1.22e-02 & 1.00 & 1.22e-03 & 1.77 & 6.49e-02 & 0.99 & 1.22e-02 & 1.00 & 1.02e-03 & 1.98 \\\hline
16384 & 3.36e-02 & 0.99 & 6.10e-03 & 1.00 & 3.89e-04 & 1.65 & 3.27e-02 & 0.99 & 6.10e-03 & 1.00 & 2.57e-04 & 1.99 \\\hline
65536 & 1.69e-02 & 0.99 & 3.05e-03 & 1.00 & 1.32e-04 & 1.56 & 1.65e-02 & 0.99 & 3.05e-03 & 1.00 & 6.47e-05 & 1.99 \\\hline
262144 & 8.51e-03 & 0.99 & 1.52e-03 & 1.00 & 4.60e-05 & 1.52 & 8.30e-03 & 0.99 & 1.52e-03 & 1.00 & 1.63e-05 & 1.99 \\\hline
  \end{tabular}
}
\end{center}
\caption{Comparison between the two mixed FEMs for the problem described in Section~\ref{sec:ex:mixed:L2}.}
  \label{tab:mixed:L2}
\end{table}

\subsection{Comparison of standard and regularized FOSLS}\label{sec:ex:fosls}
We consider the \emph{Waterfall benchmark} problem from~\cite[Section~5.4]{BringmannComp22} with manufactured solution
\begin{align*}
  u(x,y) = x(x-1)y(y-1)e^{-100(x-1/2)^2-(y-117)^2/10000}, \quad (x,y)\in \Omega:=(0,1)^2.
\end{align*}
We note that $f:=-\Delta u\in L^2(\Omega)$, and by elliptic regularity~\eqref{eq:ellipticreg} we have $u\in H^2(\Omega)\cap H_0^1(\Omega)$.
For this experiment we compare the accuracy of the standard FOSLS~\eqref{eq:fosls:standard} and the modified FOSLS~\eqref{eq:fosls:mod} with $Q_\TT^\star = \QCLw_\TT$.
Table~\ref{tab:fosls} shows the errors $\norm{\ssigma-\ssigma_\TT}{}$, $\norm{u-u_\TT}{}$, and $\norm{u-u_\TT}1$ for both methods. 
As expected $\norm{\ssigma-\ssigma_\TT}{}$ and $\norm{u-u_\TT}1$ converge at the optimal rate. From Table~\ref{tab:fosls} we even see that the absolute values of $\norm{\ssigma-\ssigma_\TT}{}$ and $\norm{u-u_\TT}1$ for both methods are not distinguishable as $h\to 0$.
Only for the $L^2(\Omega)$ errors $\norm{u-u_\TT}{}$ we see a notable difference. 
Asymptotically, one finds that the $L^2(\Omega)$ error in the primal variable for the standard FOSLS is about $42\%$ larger than the $L^2(\Omega)$ error for the regularized FOSLS.

\begin{table}[htbp]
  \begin{center}
  {\tiny
  \begin{tabular}{|c|cc|cc|cc|cc|cc|cc|}
    \hline
    \multicolumn{1}{|c|}{} &\multicolumn{6}{|c|}{regularized FOSLS~\eqref{eq:fosls:mod}} &\multicolumn{6}{|c|}{standard FOSLS~\eqref{eq:fosls:standard}}\\
     \cline{2-13}
     $\#\TT$ & $\norm{\ssigma-\ssigma_\TT}{}$ & eoc & $\norm{u-u_\TT}{}$ & eoc & $\norm{u-u_\TT}{1}$ & eoc & $\norm{\ssigma-\ssigma_\TT}{}$ & eoc & $\norm{u-u_\TT}{}$ & eoc & $\norm{u-u_\TT}{1}$ & eoc \\\hline
4 & 3.37e-02 & --- & 4.43e-03 & --- & 3.41e-02 & --- & 3.79e-02 & --- & 5.11e-03 & --- & 3.64e-02 & --- \\\hline
16 & 4.08e-02 & -0.27 & 4.19e-03 & 0.08 & 4.44e-02 & -0.38 & 3.93e-02 & -0.05 & 3.08e-03 & 0.73 & 4.23e-02 & -0.22 \\\hline
64 & 2.59e-02 & 0.65 & 1.69e-03 & 1.31 & 2.84e-02 & 0.64 & 2.60e-02 & 0.60 & 1.75e-03 & 0.81 & 2.96e-02 & 0.52 \\\hline
256 & 1.47e-02 & 0.81 & 5.37e-04 & 1.66 & 1.59e-02 & 0.84 & 1.51e-02 & 0.78 & 7.22e-04 & 1.28 & 1.72e-02 & 0.78 \\\hline
1024 & 7.43e-03 & 0.99 & 1.85e-04 & 1.54 & 1.08e-02 & 0.56 & 7.33e-03 & 1.04 & 2.41e-04 & 1.58 & 1.09e-02 & 0.66 \\\hline
4096 & 3.68e-03 & 1.01 & 4.56e-05 & 2.02 & 5.49e-03 & 0.98 & 3.66e-03 & 1.00 & 6.32e-05 & 1.93 & 5.51e-03 & 0.99 \\\hline
16384 & 1.83e-03 & 1.01 & 1.13e-05 & 2.01 & 2.76e-03 & 0.99 & 1.83e-03 & 1.00 & 1.60e-05 & 1.98 & 2.76e-03 & 1.00 \\\hline
65536 & 9.16e-04 & 1.00 & 2.82e-06 & 2.00 & 1.38e-03 & 1.00 & 9.15e-04 & 1.00 & 4.01e-06 & 2.00 & 1.38e-03 & 1.00 \\\hline
262144 & 4.58e-04 & 1.00 & 7.05e-07 & 2.00 & 6.90e-04 & 1.00 & 4.58e-04 & 1.00 & 1.00e-06 & 2.00 & 6.90e-04 & 1.00 \\\hline
   \end{tabular}
}
\end{center}
\caption{Comparison between the two FOSLS for the problem described in Section~\ref{sec:ex:fosls}.}
  \label{tab:fosls}
\end{table}

\appendix
\section{Further comments on~\eqref{eq:comb}}\label{sec:convexcomb}
In this section we study under which conditions we can ensure that the coefficients in~\eqref{eq:comb} are non-negative, i.e., 
\begin{align}\label{eq:convexcomb}
  \sum_{T\in\patch_z} \alpha_{z,T}  s_T = z, \quad\sum_{T\in\patch_z} \alpha_{z,T} = 1, 
  \quad \alpha_{z,T}\geq 0 \quad(T\in\patch_z).
\end{align}
Here and throughout this section, $z\in \VV_0$ denotes an interior vertex. Note that~\eqref{eq:convexcomb} states that $z$ can be written as a convex combination of the centers of mass $\{s_T\,:\,T\in\patch_z\}$.
In the recent work~\cite{KMZ26} the authors point out that~\eqref{eq:convexcomb} is guaranteed if no mesh elements have obtuse angles~\cite[Remark~6]{KMZ26}.
The non-negativity of the coefficients then implies that the weighted Cl\'ement operator preserves non-negativity, that is, $\varphi_z\geq 0$ for all $z\in\VV_0$, and, therefore,
\begin{align*}
  \JCLw_\TT v = \sum_{z\in\VV_0} \ip{v}{\varphi_z} \eta_z \geq 0 \quad\text{if } v\geq 0.
\end{align*}
Such a property is interesting for the analysis of variational inequalities, see, e.g.~\cite{KMZ26}.
We note that~\cite{KMZ26} construct an operator with non-negative weights by using higher-order patches.
We stress that none of the results in Section~\ref{sec:main}--\ref{sec:ex} needs~\eqref{eq:convexcomb}.

Recall from Section~\ref{sec:clement} that any $y\in\R^n$ can be written as
\begin{align*}
  y =   \sum_{T\in\patch_z} \beta_{z,T}(y)s_T,
\end{align*}
where $\beta_{z,T}(y) := \frac{|T|}{|\Patch_z|} (1+(n+1)(s_z-y)\cdot\nabla\eta_z|_T)$ for all $T\in\patch_z$.

For $y=z$ we can rewrite the coefficient for $T\in\patch_z$ as
\begin{align*}
  \beta_{z,T}(z) = \frac{|T|}{|\Patch_z|} (1+(n+1)(s_z-z)\cdot\nabla\eta_z|_T) = 
  \frac{|T|}{|\Patch_z|}((n+1)\eta_{z,T}(s_z)-n)
\end{align*}
where $\eta_{z,T}(x) := \nabla\eta_z|_T\cdot(x-z) + 1$, $x\in\R^n$. Note that $\eta_{z,T}(x) = \eta_z(x)$ for all $x\in T$.

\begin{theorem}\label{thm:simplecondition}
  If $\eta_{z,T}(s_z) \geq \frac{n}{n+1}$ for all $T\in\patch_z$, then $z\in\conv\{s_T\,:\,T\in\patch_z\}$.

  If $\Patch_z$ is convex, then $\eta_z(s_z)\geq \frac{n}{n+1}$ implies that $z\in\conv\{s_T\,:\,T\in\patch_z\}$.
\end{theorem}
\begin{proof}
  The first part directly follows from $\beta_{z,T}(z)\geq 0$ iff $(n+1)\eta_{z,T}(s_z)-n\geq 0$ and choosing $\alpha_{z,T}=\beta_{z,T}(z)$ for $T\in\patch_z$.

  Since $\Omega_z$ is convex, $s_z\in\Omega_z$, and for every $T'\in\patch_z$ the hyperplane $\{\eta_{z,T'}=0\}$ containing the face of $T'$ opposite $z$ supports $\Omega_z$, so $\eta_{z,T'}\geq 0$ on $\Omega_z$. Let $T\in\patch_z$ and $x\in T$ and write $x=z+t(p-z)$ with $t\in[0,1]$ and $p$ on the face of $T$ opposite $z$. Then, $\eta_{z,T'}(x)=1-t+t\eta_{z,T'}(p)\geq 1-t=\eta_z(x)$. Hence, $\eta_z=\min_{T'\in\patch_z} \eta_{z,T'}$ on $\Omega_z$, and in particular $\eta_{z,T'}(s_z)\geq \eta_z(s_z)\geq n/(n+1)$ for all $T'$.
\end{proof}

\begin{theorem}\label{thm:simplecondition2}
  If $(y-z)\cdot (s_T-z) \geq 0$ for all $y\in \VV_T$, $T\in\patch_z$, then $z\in\conv\{s_T\,:\,T\in\patch_z\}$.
\end{theorem}
\begin{proof}
  W.l.o.g. assume that $z$ is the origin. 
  Suppose that $z\notin C:=\conv\{s_T\,:\,T\in\patch_z\}$.
  Since $C$ is compact and convex and $z\notin C$, there is $x\neq 0$ with $x\cdot s_T < 0 = x\cdot z$ for all $T$. 
  Since $z$ is an interior point, we may scale $x$ so that $x\in\Patch_z$.
  Then, there exists $T\in\patch_z$ with $x\in T$ and $x = \sum_{y\in \VV_T\setminus\{z\}} \lambda_y y$ with $\lambda_y\geq 0$ (and not all equal to $0$). Therefore, $x\cdot s_T = \sum_{y\in\VV_T\setminus\{z\}} \lambda_y y\cdot s_T\geq 0$ which contradicts $x\cdot s_T<0$.
\end{proof}

Note that the conditions given in Theorem~\ref{thm:simplecondition} and Theorem~\ref{thm:simplecondition2} are sufficient but not necessary as can be seen from the following two examples. 

\begin{example}\label{ex:nodes1}
  Consider the same setup as in Example~\ref{ex:CL2d} but with $z = (\tfrac12,\tfrac18)$.
  Then, $\eta_z(s_z) = \tfrac47<\tfrac23$, but $(z_k-z)\cdot(s_{T_i}-z)>0$ for all $k\in\{i,i+1\}$ and $i=1,\ldots,4$.
\end{example}

\begin{example}
  Consider the vertices $z_1 = (-4,-1)$, $z_2=(8,-1)$, $z_3=(-4,2)$ and $z=(0,0)$ as well as the triangles $T_j = \conv\{z_j,z_{\!\!\mod(j,3)+1},z\}$, $j=1,2,3$. Note that $z=s_z$, hence, $\eta_z(s_z) = 1\geq \tfrac23$, but $(z_1-z)\cdot(s_{T_1}-z)=-\tfrac{14}3<0$.
\end{example}

\bibliographystyle{abbrv}
\bibliography{literature}

\end{document}